\documentclass{elsarticle}
[1p,preprint]

\usepackage[all,2cell]{xy}
\usepackage{mathtools,halloweenmath}
\usepackage{graphicx}
\usepackage{amsmath}
\usepackage{amsthm}
\usepackage{amsfonts}
\usepackage{amssymb}
\usepackage{graphicx}
\usepackage{quiver}
\usepackage[colorlinks=true, allcolors=blue]{hyperref}
\usepackage{todonotes}
\theoremstyle{plain} 
\newtheorem{theorem}{Theorem}[section]
\newtheorem{corollary}[theorem]{Corollary}
\newtheorem{lemma}[theorem]{Lemma}
\newtheorem{proposition}[theorem]{Proposition}

\theoremstyle{definition} 
\newtheorem{definition}[theorem]{Definition}
\newtheorem{remark}[theorem]{Remark}
\newtheorem{example}[theorem]{Example}

\makeatletter
\def\ifenv#1{
	\def\@tempa{#1}%
	\ifx\@tempa\@currenvir
	   \expandafter\@firstoftwo
	 \else
	   \expandafter\@secondoftwo
	\fi
}

\newcommand{\thlabel}[1]{%
	\label{#1}
	\begingroup
	\ifenv{theorem}{
		\def\@currentlabel{Theorem \ref{#1}}
		\label{#1-th}
	}
	{\ifenv{unprovedtheorem}{
		\def\@currentlabel{Theorem \ref{#1}}
		\label{#1-th}
	}
 	{\ifenv{proposition}{
		\def\@currentlabel{Proposition \ref{#1}}
		\label{#1-th}
	}
	{\ifenv{unprovedproposition}{
		\def\@currentlabel{Proposition \ref{#1}}
		\label{#1-th}
	}
	{\ifenv{lemma}{
		\def\@currentlabel{Lemma \ref{#1}}
		\label{#1-th}
	}
	{\ifenv{corollary}{
		\def\@currentlabel{Corollary \ref{#1}}
		\label{#1-th}
	}
	{\ifenv{remark}{
		\def\@currentlabel{Remark \ref{#1}}
		\label{#1-th}
	}
	{\ifenv{example}{
		\def\@currentlabel{Example \ref{#1}}
		\label{#1-th}
	}
	{\ifenv{definition}{
		\def\@currentlabel{Definition \ref{#1}}
		\label{#1-th}
	}
	{\GenericWarning{No Environment Found}}}}}}}}}}
	\endgroup
}

\newcommand{\thref}[1]{\ref{#1-th}}

\makeatother

\newcommand{\ovln}[1]{\overline{#1}}

\newcommand{\Excat}{\mathsf{Ex}}
\newcommand{\exlex}[1]{(#1)_{\mathsf{ex}/\mathsf{lex}}}

\newcommand{\angbr}[2]{\langle #1,#2 \rangle}

\newcommand{\freccia}[3]{#2 \colon #1  \to   #3}
\newcommand{\arrow}[3]{#2 \colon #1  \longrightarrow #3}

\newcommand{\frecciainj}[3]{\xymatrix{#2 \colon #1  \ar@{^{(}->}[r] &  #3}}

\newcommand{\arrowup}[3]{#1 \xrightarrow{\; #2 \;} #3}

\newcommand{\pbmorph}[2]{#1^{\ast}#2} 

\newcommand{\duefreccia}[3]{\xymatrix@C=0.5cm{#2 \colon #1  \ar@{=>}[r] &  #3}}

\newcommand{\comsquare}[8]{ \xymatrix@+1pc{ 
#1 \ar[r]^{#5} \ar[d]_{#6} & #2 \ar[d]^{#7} \\
#3 \ar[r]_{#8} & #4 
}}
\newcommand{\pullback}[8]{ \xymatrix@+1pc{ 
#1 \pullbackcorner \ar[r]^{#5} \ar[d]_{#6} & #2 \ar[d]^{#7} \\
#3 \ar[r]_{#8} & #4 
}}
\newcommand{\quadratocomm}[8]{ \xymatrix@+1pc{ 
#1 \ar[r]^{#5} \ar[d]_{#6} & #2 \ar[d]^{#7} \\
#3 \ar[r]_{#8} & #4 
}}
\newcommand{\comsquarelargo}[8]{ \xymatrix@+1pc{ 
#1 \ar[rr]^{#5} \ar[d]_{#6} && #2 \ar[d]^{#7} \\
#3 \ar[rr]_{#8} && #4 
}}
\newcommand{\parallelmorphisms}[4]{\xymatrix@+1pc{
#1 \ar @<+4pt>[r]^{#2} \ar @<-4pt>[r]_{#3} & #4
}}
\newcommand{\relation}[4]{\xymatrix@+1pc{
\angbr{#2}{#3}\colon #1 \ar @<+4pt>[r] \ar @<-4pt>[r] & #4
}}
\newcommand{\frecceparalleleopposte}[4]{\xymatrix@+1pc{
#1 \ar@<+4pt>[r]^{#2} \ar@<-4pt>@{<-}[r]_{#3} & #4
}}
\newcommand{\equalizer}[6]{\xymatrix@+1pc{
#1 \ar[r]^{#2} & #3 \ar @<+4pt>[r]^{#4} \ar @<-4pt>[r]_{#5} & #6
}}
\newcommand{\coequalizer}[6]{\xymatrix@+1pc{
 #1 \ar @<+4pt>[r]^{#2} \ar @<-4pt>[r]_{#3} & #4 \ar[r]^{#5} & #6
}}

\newcommand{\sottoggetto}[2]{\xymatrix{
#1 \ar@{>->}[r] & #2
}}

\newcommand{\pullbackcorner}[1][ul]{\save*!/#1+1.2pc/#1:(1,-1)@^{|-}\restore}

\def\pr{\pi}
\def\id{\operatorname{ id}}
\def\op{\operatorname{ op}}

\def\dom{\operatorname{ dom}}

\def\mC{\mathcal{C}}

\def\mD{\mathcal{D}}

\def\Sub{\mathsf{Sub}}

\newcommand{\pca}[1]{\mathbb{#1}}
\newcommand{\subpca}[1]{\mathbb{#1}'}

\def\set{\mathsf{Set}}

\newcommand{\PD}{\mathsf{PD}}
\newcommand{\ED}{\mathsf{ED}}

\newcommand{\doctrine}[2]{#2 \colon #1^{\op}  \longrightarrow  \mathsf{InfSl} }

\newcommand{\function}[2]{\colon #1 \to #2}
\newcommand{\pfunction}[2]{:\subseteq #1 \to #2}

\newcommand{\ran}{\operatorname{ran}}

\newcommand{\powerset}{\raisebox{0.6mm}{\Large\ensuremath{\wp}}}

\newcommand{\Mpcadoctrine}{\mathfrak{T}}
\newcommand{\Mpcaorder}{\le_{\mathsf{T}}}

\newcommand{\MedvedevOrder}{\le_\mathrm{M}}
\newcommand{\MedvedevDoctrine}{\mathfrak{M}}

\newcommand{\barAsm}[1]{|#1|}
\newcommand{\relAsm}[1]{\Vdash_{#1}}

\newcommand{\extweireducible}{\le_{\mathsf{extW}}}

\newcommand{\EWeiElementary}{\mathfrak{eiR}}
\newcommand{\InstRedDoc}{\mathfrak{iR}}
\newcommand{\EWeiOrderElementary}{\le_\mathsf{eI}}

\newcommand{\EWeiDoctrineOrder}{\le_\mathsf{extW}}
\newcommand{\EWei}{\mathfrak{eW}}

\newcommand{\Assemblies}{\mathsf{Asm}}

\newcommand{\str}[1]{( #1 )}
\newcommand{\pairing}[1]{\langle #1 \rangle}
\newcommand{\coding}[1]{\langle #1 \rangle}
\newcommand{\concat}{\smash{\raisebox{.9ex}{\ensuremath\smallfrown} }}

\newcommand{\Baire}{\mathbb{N}^\mathbb{N}}

\newcommand{\partAsm}{\mathsf{ParAsm}}

\newcommand{\fprPCA}{\mathsf{p}_1}
\newcommand{\sprPCA}{\mathsf{p}_2}
\newcommand{\kPCA}{\mathsf{k}}
\newcommand{\sPCA}{\mathsf{s}}
\newcommand{\bkPCA}{\bar{\kPCA}}

\newcommand{\pair}{\mathsf{pair}}
\newcommand{\name}{\Vdash}

\usepackage{tikz-cd}

\usepackage{url}
\usepackage{doi}

\title{A topos for extended Weihrauch degrees}
\author{Samuele Maschio and Davide Trotta}

\begin{document}

\begin{abstract}
    Weihrauch reducibility is a notion of reducibility between computational problems that is useful to calibrate the uniform computational strength of a multi-valued function.
It complements the analysis of mathematical theorems done in reverse mathematics, as multi-valued functions on represented spaces can be considered as realizers of theorems in a natural way. Despite the rich literature and the relevance of the applications of category theory in logic and realizability, actually there are just a few works starting to study the Weihrauch reducibility from a categorical point of view. The main purpose of this work is to provide a full categorical account of the notion of
extended Weihrauch reducibility introduced by A. Bauer, which generalizes the original notion
of Weihrauch reducibility. In particular, we present a tripos and a topos for extended Weihrauch degrees. We start by defining a new tripos, abstracting the notion of extended Weihrauch degrees,
and then we apply the tripos-to-topos construction to obtain the desired topos. Then we show that the Kleene-Vesley topos is a topos of $j$-sheaves for a certain Lawvere-Tierney topology over the
topos of extended Weihrauch degrees. 
\end{abstract}
\maketitle
\tableofcontents
\section{Introduction}
Categorical logic is the branch of mathematics in which tools and concepts from category theory are applied to the study of mathematical logic.
Nowadays, categorical methods are applied in many areas of mathematical logic. One area that has been less influenced by this categorical approach is computability and, in light of the long tradition of the application of categorical methods in realizability studies, see e.g. \cite{van_Oosten_realizability}, this seems to be quite peculiar. 
In recent years, there have been some works starting to approach computability-like notions from a categorical perspective. E.g.\ a first approach to Medvedev and Muchnik reducibility via hyperdoctrines has been introduced in \cite{Kuyper2015}, while Weihrauch reducibility for assemblies (or multi-represented spaces) has been introduced only very recently in \cite{Kihara2022rethinking, SchroederCCA2022} through the notion of realizer-based Weihrauch reducibility. In \cite{Bauer2021}, Bauer introduced an abstract notion of reducibility between predicates, called instance reducibility, which commonly appears in reverse constructive mathematics. In a relative realizability topos, the instance degrees correspond to a generalization of (realizer-based) Weihrauch reducibility, called \emph{extended Weihrauch degrees}, and the ``classical'' Weihrauch degrees \cite{BGP17} correspond precisely to the $\lnot\lnot$-dense modest instance degrees in Kleene-Vesley realizability. Upon closer inspection, it is not hard to check that realizer-based Weihrauch reducibility is a particular case of Bauer's notion.

Finally, in \cite{trottavalenti2025},  a categorical formulation of Turing, Medvedev, Muchnik, and Weihrauch reducibilities is presented using Lawvere doctrines, which are a generalization of the notion of \emph{hyperdoctrine} introduced by Lawvere ~\cite{lawvere1969,lawvere1969b,lawvere1970} to synthesize the structural properties of logical systems.

The main purpose of this work is to carry on this research line, focusing on the notion of extended Weihrauch degrees. Our main goal is to define a topos for extended Weihrauch degrees, providing a suitable universe for studying this reducibility categorically, in the same way as Hyland's Effective topos~\cite{hyland1982effective} is a suitable universe for studying realizability. Then, we take advantage of this categorical presentation, and we establish the precise connection between extended Weihrauch degrees and realizability. 

In detail: the main tools we adopt to define and construct such a topos is the so-called \emph{tripos-to-topos construction} introduced by Hyland, Johnstone, and Pitts  \cite{hyland89,pitts02}, producing a topos from a given tripos, and the \emph{(full) existential completion}, a construction that freely adds left adjoint, i.e.\ existential quantifiers, along all the morphisms of the base of a given doctrine, for defining the tripos abstracting extended Weihrauch degrees. Recall that a tripos is a specific instance of the notion of Lawvere hyperdoctrine, which has enough structure to deal with higher-order logic properly. The main reason for considering extended Weihrauch degrees is that they provide the structure of a tripos. This is not guaranteed by either of the intermediate notions: the usual Weihrauch degrees (i.e.\ $\neg\neg$-dense instance degrees) fail to yield Heyting algebras in the fibres \cite{HiguchiPauly}, while modest instance degrees \cite{Bauer2021} do form Heyting algebras but lack a generic predicate. Therefore, extended Weihrauch degrees are precisely what is needed to obtain a tripos.

Hence, we first define a doctrine $\doctrine{\partAsm(\pca{A},\subpca{A})}{\InstRedDoc}$ abstracting Bauer's notion of instance reducibility between realizability predicates, and we prove that it is a tripos (\thref{thm_tripos_ex_wei}). This doctrine is defined as the (full) existential completion  $\doctrine{\partAsm(\pca{A},\subpca{A})}{\EWeiElementary^{\exists}}$ of a more basic doctrine $\doctrine{\partAsm(\pca{A},\subpca{A})}{\EWeiElementary}$, following the same idea used in \cite{trottavalenti2025} for defining doctrines abstracting computability reducibility. 


Then we introduce a second doctrine $\doctrine{\partAsm(\pca{A},\subpca{A})}{\EWei}$ (\thref{def_ex_deg_doctrine}) which provides a direct categorification of  the notion of extended Weihrauch degrees and we prove that $\EWei\cong \InstRedDoc$ (\thref{thm_ex_descr_wei_doc}). This equivalence shows in particular that $\doctrine{\partAsm(\pca{A},\subpca{A})}{\EWei}$ is a tripos.

Notice that this result can be seen as a fibrational version of Bauer's result showing that  Weihrauch reductions and instance reductions are equivalent, see~\cite[Prop. 8]{Bauer2021}.  


Finally, once we have proved that the extended Weihrauch doctrine is a tripos, we can define the topos of extended Weihrauch degrees as the topos obtained by applying the tripos-to-topos to the tripos $\EWei$ and study the connections with (relative) realizability toposes. In particular, we prove that  the relative realizability topos $\mathsf{RT}[\pca{A},\pca{A}'] $ \cite{BIRKEDAL2002115} is equivalent to a topos $\mathsf{sh}_j(\mathsf{EW}[\pca{A},\pca{A}'])$ of $j$-sheaves  for a certain Lawvere-Tierney topology $j$ over $ \mathsf{EW}[\pca{A},\pca{A}']$ (\thref{cor_sheaves_real_top}). To prove this result, there are two factors playing a key role: first that the extended Weihrauch tripos is a (full) existential completion; second, that realizability toposes can be presented as exact completions of the category of partitioned assemblies~\cite{robinsonrosolini90}.  Both of these results rely on the use of the Axiom of Choice, which we will assume throughout this paper.

\section{Preliminaries}\label{sec_prel}
\subsection{Partial combinatory algebras
}\thlabel{sec:basic notions}

Realizability theory originated with Kleene’s interpretation of intuitionistic number theory~\cite{kleene1945} and has since then developed into a large body of work in logic and theoretical computer science. We focus on two basic flavors of realizability, number realizability, and function realizability, which were both due to Kleene. 

In this section, we recall some standard notions within realizability and computability. We follow the approach suggested by van Oosten~\cite{van_Oosten_realizability}, as we want to fix a suitable notation for both category theorists and computability logicians. 

We describe \emph{partial combinatory algebras} and discuss some important examples.  For more details, we refer the reader to van Oosten's work on categorical realizability (see \cite{van_Oosten_realizability} and the references therein). 

We start by introducing the basic concept of a \emph{partial applicative structure} or PAS, due to Feferman, which may be viewed as a universe for computation. 

\begin{definition}[PAS]
    A  \textbf{partial applicative structure}, or PAS for short, is a set $\pca{A}$ equipped with a 
    \textbf{partial} binary operation 
$\cdot\pfunction{\pca{A}\times \pca{A}}{\pca{A}}$.
\end{definition}

Some conventions and terminology: given two elements $a,b$ in $\pca{A}$, we think of $a\cdot b$ (which we often abbreviate $ab$) as  ``$a$ applied to $b$". The partiality of the operation $\cdot$ means that this application need not always be defined. We write $f\pfunction{A}{B}$ to say that $f$ is a partial function whose domain is a subset of $A$, $\dom(f)\subseteq A$, and range $\ran(f)\subseteq B$.
If $(a,b)\in \dom(\cdot)$, that is, if the application is defined, then we write $a\cdot b \downarrow$ or $ab\downarrow$. 

We usually omit brackets, assuming the associativity of application to the left. Thus $abc$ stands for $(ab)c$. Moreover, for two expressions $s$ and $s'$ we write $s \simeq  s'$ to indicate that $s$ is defined exactly when $s'$ is, in which case they are equal.

Even though these partial applicative structures do not possess many interesting properties (they have no axioms for application), they already highlight one of the key features of combinatorial structures, namely the fact that we have a domain of elements that can act both as functions and as arguments, just as in untyped $\lambda$-calculus. 
This behaviour can be traced back to von Neumann's idea that programs (functions, operations) live in the same realm and are represented in the same way as the data (arguments) that they act upon.
In particular, programs can act on other programs.

\begin{definition}[PCA]\thlabel{def: PCA}
    A \textbf{partial combinatory algebra} (PCA) is a PAS $\pca{A}$  for which there exist elements $\kPCA,\sPCA\in \pca{A}$ such that for all $a,b,c\in \pca{A}$ we have that
    \begin{enumerate}
    \item $\kPCA ab=a$
    \item $\sPCA ab \downarrow$
    \item $\sPCA abc \simeq ac(bc)$
    \end{enumerate}
\end{definition}

The elements $\kPCA$ and $\sPCA$ are generalizations of the homonymous combinators in Combinatory Logic. 

Every PCA $\pca{A}$ is combinatory complete in the sense of \cite{Feferman75,van_Oosten_realizability}, namely: for every term $t(x_1,\dots,x_{n+1})$ built from variables $x_1,\dots,x_{n+1}$, constants $c\in \pca{A}$, and application operator $\cdot$, there exists an element $a\in \pca{A}$ such that for all elements $b_1,\dots,b_{n+1}\in \pca{A}$ we have that $ab_1\cdots b_n \downarrow$ and $ab_1\cdots b_{n+1}\simeq t(b_1,\dots,b_{n+1})$. Such an element $a$ can be taken as $\lambda x_1...\lambda x_{n+1}.t$, once one has encoded $\lambda$-abstraction as follows:
\begin{enumerate}
\item $\lambda x.x:=\sPCA\kPCA\kPCA$;
\item $\lambda x.y:=\kPCA y$ if $y$ is a variable different from $x$;
\item $\lambda x.a:=\kPCA a$ if $a\in \mathbb{A}$;
\item $\lambda x.tt'=\sPCA(\lambda x.t)(\lambda x.t')$.
\end{enumerate}

In particular, we can use this result and the elements $\kPCA$ and $\sPCA$ to construct elements $\pair,  \fprPCA, \sprPCA $ of $\pca{A}$ so that $(a,b)\mapsto \pair \cdot a\cdot b$ is an injection of $\pca{A}\times \pca{A}$ into $\pca{A} $ with left inverse $c\mapsto (\fprPCA \cdot c,\sprPCA \cdot c)$. Hence, we can use $\pair ab$ as an element of $\pca{A}$ which encodes the pair $(a, b)$. For this reason, the elements $\pair, \fprPCA ,\sprPCA $ are usually called \emph{pairing} and \emph{projection} operators. For the sake of readability, we write $\pairing{a,b}$ in place of $\pair ab$ (as it is more customary in computability theory).

Using $\kPCA$ and $\sPCA$, we can prove the analogues of the Universal Turing Machine (UTM) and the SMN theorems in computability in an arbitrary PCA. 

If $X,Y\subseteq \mathbb{A}$, we denote with $X\otimes Y$ the set $\{\pairing{ x,y}|\,x\in X,y\in Y\}$ and with  $X\oplus Y$ the set $(\{\kPCA\}\otimes X)\cup (\{\bkPCA\}\otimes Y)$, where $\bkPCA$ is $\kPCA(\sPCA\kPCA\kPCA)$.

Next we recall the notion of \emph{elementary sub-PCA} (see for example \cite[Sec.\ 2.6.9]{van_Oosten_realizability}). 

\begin{definition}[elementary sub-PCA] \thlabel{def:sub-pca}
Let $\pca{A}$ be a PCA. A subset $\subpca{A}\subseteq \pca{A}$ is called an \textbf{elementary sub-PCA} of $\pca{A}$ if $\subpca{A}$ is a PCA with the partial applicative structure induced by $\pca{A}$, and the elements $\kPCA$ and $\sPCA$ as in \thref{def: PCA} can be found in $\pca{A}'$. In particular, it is closed under the application of $\pca{A}$, that means: if $a,b\in \subpca{A}$ and $ab\downarrow$ in $\pca{A}$ then $ab\in \subpca{A}$.
\end{definition}

In particular, elements of the sub-PCA will play the role of the computable functions.
From now on, when dealing with a PCA $\pca{A}$ with an elementary sub-PCA $\subpca{A}$, we will implicitly assume $\kPCA$ and $\sPCA$ to be in $\subpca{A}$, and that these are used to produce the derived combinators and lambda terms presented above.

\begin{example}[\emph{Kleene’s first model}]\label{ex: K_1}
Fix an effective enumeration $(\varphi_a)_{a\in\mathbb{N}}$ of the partial recursive functions $\mathbb{N}\to\mathbb{N}$ (i.e.\ a G\"odel numbering). The set $\pca{N}$ with partial recursive application $(a,b)\mapsto \varphi_a(b)$ is a PCA, and it is called \emph{Kleene’s first model} $\mathcal{K}_1$ (see e.g.\ \cite{Soare1987}).
\end{example}
The following example of a PCA is arguably the most relevant for computable analysis and the most extensively studied in the context of Weihrauch reducibility:
\begin{example}
[\emph{Kleene’s second model}] The PCA $\mathcal{K}_2$ is often used for function realizability \cite[Sec.\ 1.4.3]{van_Oosten_realizability}. This PCA is given by the Baire space $\pca{N}^{\pca{N}}$, endowed with the product topology. 
The partial binary operation of application $\cdot\pfunction{\pca{N}^{\pca{N}}\times \pca{N}^{\pca{N}}}{\pca{N}^{\pca{N}}} $ corresponds to the one used in Type-$2$ Theory of Effectivity \cite{Weihrauch00}.
This can be described as follows. Let $\alpha[n]$ denote the string $\str{\alpha(0),\hdots, \alpha(n-1)}$. 
Every $\alpha\in \Baire$ induces a function $F_{\alpha}\pfunction{\Baire}{\mathbb{N}}$ defined as $F_{\alpha}(\beta)=k$ if there is $n\in\mathbb{N}$ such that $\alpha(\coding{\beta[n]})=k+1$ and $(\forall m<n)(\alpha(\coding{\beta[m]})=0)$, and undefined otherwise. The application $\alpha\cdot \beta$ can then be defined as the map $n\mapsto F_{\alpha}(\str{n}\concat \beta)$, where $\str{n}\concat \beta$ is the string $\sigma$ defined as $\sigma(0):=n$ and $\sigma(k+1):=\beta(k)$.

When working with Kleene's second model, we usually consider the elementary sub-PCA $\mathcal{K}_2^{rec}$ consisting of all the $\alpha\in\Baire$ such that $\beta\mapsto \alpha\cdot \beta$ is computable. For more details and examples, also of non-elementary sub-PCAs for Kleene's second model, we refer to \cite{Oosten2011PartialCA}.
\end{example}


\subsection{Assemblies and partitioned assemblies}\label{sec:rep_spaces}
We briefly recall some useful definitions regarding the notion of \emph{assembly} and \emph{partitioned assembly} (in relative realizability), explaining how these can be seen as a categorification (and a generalization) of the notion of represented space. We refer to \cite{van_Oosten_realizability} for a complete presentation of these notions, and to \cite{Bauer2021,Frey2014AFS,HYLAND1988,BIRKEDAL2002115} for more specific applications. In the following definition, we mainly follow the notation used in \cite{Bauer2021}, which is closer to the usual one used in computability.

\begin{definition}[Assemblies and Partitioned assemblies]\thlabel{def:assemblies}
Let $\pca{A}$ be a PCA. An \textbf{assembly} is a pair $X=(\barAsm{X},\relAsm{X})$ where $\barAsm{X}$ is a set and $\relAsm{X}\subseteq \pca{A}\times \barAsm{X}$ is a total relation, 
i.e.\ $(\forall x\in \barAsm{X})(\exists r\in \pca{A})( r\relAsm{X}x)$. 
An assembly $X$  is \textbf{partitioned} if  $r\relAsm{X}x$ and  $r'\relAsm{X}x$ imply $r=r'$, i.e.\ every element of $X$ has exactly one name in $\mathbb{A}$.

\end{definition}


\begin{definition}[Morphism of assemblies]
Let $\pca{A}$ be a PCA and $\subpca{A}\subseteq \pca{A}$ be an elementary sub-PCA of $\pca{A}$.
A \textbf{morphism of assemblies} $f\function{X}{Y}$ is a function $f\function{\barAsm{X}}{\barAsm{Y}}$ which has a realizer in $\pca{A}'$, i.e.\ it is a function such that there exists an element $a\in \pca{A}'$ with the property that for every $r\relAsm{X}x$ it holds that $a\cdot r\downarrow$ and $a\cdot r\relAsm{Y} f(x)$.
\end{definition}

\begin{definition}
Let $\pca{A}$ be a PCA and $\subpca{A}$ be an elementary sub-PCA of $\pca{A}$. We define the category of \textbf{assemblies} $\Assemblies(\pca{A},\subpca{A})$ as the category 
having assemblies as objects and where arrows are morphisms of assemblies (with the usual set-theoretical composition of functions). 
The category of \textbf{partitioned assemblies} $\partAsm(\pca{A},\subpca{A})$ is the full subcategory of $\Assemblies(\pca{A},\subpca{A})$ whose objects are partitioned assemblies.
\end{definition}

\subsection{Primary and existential doctrines}
Lawvere introduced the notion of \emph{hyperdoctrine}~\cite{lawvere1969,lawvere1969b,lawvere1970} to synthesize in a categorical framework the structural properties of logical systems.

Several generalizations of such a notion have been considered recently, e.g. we refer to the works of Rosolini and Maietti \cite{maiettipasqualirosolini,maiettirosolini13a,maiettirosolini13b}. 

We start by recalling a simple generalization of the notion of hyperdoctrine:

\begin{definition}[primary doctrine]\label{def primary doctrine}
A \textbf{primary doctrine}~is a functor
$$\doctrine{\mC}{P}$$ from the opposite of a category $\mC$ with finite limits to the category $\mathsf{InfSl}$ of inf-semilattices. 
\end{definition}

Notice that, with respect to the original definition of primary doctrine, here we also require that the base category has finite limits.

\begin{example}[\cite{Hofstra2006AllRI}]\label{ex: primary doctrine pca}
Let $\pca{A}$ be a partial combinatory algebra (PCA). We can define a functor $\doctrine{\mathsf{Set}}{\pca{A}^{(-)}}$ assigning to a set $X$ the set $\pca{A}^X$ of functions from $X$ to $\pca{A}$. Given two elements $\alpha,\beta\in \pca{A}^X$,  we have that $\alpha\leq \beta$ if there exists an element $a\in \pca{A}$ such that for every $x\in X$ we have that $a\cdot \alpha(x)$ is defined and $a\cdot \alpha(x)=\beta(x)$.
\end{example}
\begin{example}[\cite{maiettirosolini13a}]
    Let $\mC$ be a category with finite limits. We can define the primary doctrine of weak subobjects $\doctrine{\mC}{\Psi}$ as the functor sending an object $X$ into the poset reflection of the category $\mC/X$ and acting by pullback on arrows. 
\end{example}

\begin{definition}[morphism of primary doctrines]\label{def:morphism of primary doctrine}
Let $\doctrine{\mC}{P}$ and $\doctrine{\mD}{R}$ be two primary doctrines. A \textbf{morphism} of primary doctrines is given by a pair $(F,\mathfrak{b})$ 

\[\begin{tikzcd}
   \mC^{\op} \\
   && \mathsf{InfSl} \\
   \mD^{\op}
   \arrow[""{name=0, anchor=center, inner sep=0}, "R"', from=3-1, to=2-3]
   \arrow[""{name=1, anchor=center, inner sep=0}, "P", from=1-1, to=2-3]
   \arrow["F^{\op}"', from=1-1, to=3-1]
   \arrow["\mathfrak{b}"', shorten <=4pt, shorten >=4pt, from=1, to=0]
\end{tikzcd}\]
where
\begin{itemize}
\item $\freccia{\mC}{F}{\mD}$ is a finite limits preserving functor;
\item $\freccia{ P}{\mathfrak{b}}{R\circ F^{\op}}$ is a natural transformation.
\end{itemize}
\end{definition}
\begin{definition}[doctrine transformation]\label{def:doctrine transformation}
Let $\doctrine{\mC}{P}$ and $\doctrine{\mD}{R}$ be two primary doctrines, $\freccia{P}{(F,\mathfrak{b}),(G,\mathfrak{c})}{R}$ be two primary morphisms of doctrines. A \textbf{doctrine transformation} is a natural transformation $\freccia{F}{\theta}{G}$ such that
\[ \mathfrak{b}_A(\alpha)\leq R_{\theta_A}(\mathfrak{c}_A(\alpha))\]
for every $\alpha$ of $P(A)$.
\end{definition}
Primary doctrines, morphisms of primary doctrines and doctrine transformations form a 2-category denoted by $\PD$.
\begin{definition}[existential doctrine]\label{def existential doctrine}
A primary doctrine $\doctrine{\mC}{P}$ is \textbf{(full) existential} if, for every object $A$ and $B$ in $\mC$ and for any arrow $\freccia{A}{f}{B}$, the functor
\[ \freccia{P(B)}{{P_{f}}}{P(A)}\]
has a left adjoint $\exists_{f}$, and these satisfy:

\begin{itemize}
    \item[(BCC)] the \textbf{Beck-Chevalley condition:} for any pullback diagram
\[\begin{tikzcd}[column sep=large, row sep=large]
	{A'} & {B'} \\
	A & B
	\arrow["{g'}"', from=1-1, to=2-1]
	\arrow["f"', from=2-1, to=2-2]
	\arrow["{f'}", from=1-1, to=1-2]
	\arrow["g", from=1-2, to=2-2]
	\arrow["\scalebox{1.6}{$\lrcorner$}"{anchor=center, pos=0.1}, shift left=3, draw=none, from=1-1, to=2-1]\end{tikzcd}\]
the canonical arrow
\[ \exists_{f'}P_{g'}(\alpha)\leq P_g \exists_{f}(\alpha)\]
is the identity for every element $\alpha$ of the fibre $P(A)$. 
\item[(FR)] \textbf{Frobenius reciprocity:} for any object $\alpha$ in $P(B)$ and $\beta$ in $P(A)$, the canonical arrow
\[ \exists_{f}(P_{f}(\alpha)\wedge \beta)\leq \alpha \wedge \exists_{f}(\beta)\]
in $P(A)$ is the identity.
\end{itemize}
\end{definition}
For the sake of clarity, in this work, we will use the adjective ``existential" for doctrines having left adjoints along \emph{all} the morphisms of the base category, satisfying (BCC). Hence, this is a particular case of the notion of elementary and existential doctrine used by Maietti and Rosolini in~\cite{maiettirosolini13a,maiettirosolini13b}.

\begin{definition}[existential morphism]
Let $\doctrine{\mC}{P}$ and $\doctrine{\mD}{R}$ be existential doctrines.
A morphism  of primary doctrines $(F,\mathfrak{b})$ is said to be \textbf{existential} if for every arrow $\freccia{A}{f}{B}$ of $\mC$ we have that 
$$\exists_{Ff}\mathfrak{b}_A(\alpha)=\mathfrak{b}_B(\exists_f(\alpha))$$ 
for every element $\alpha$ of $P(A)$.
\end{definition}
We denote the 2-category of existential doctrines, existential morphisms and doctrine transformations by $\ED$.
\begin{definition}[first order hyperdoctrine]\label{def:hyperdoctrine}
A \textbf{first order hyperdoctrine} is an existential doctrine $\doctrine{\mC}{P}$  such that
\begin{itemize}
\item for every object $A$ of $\mC$ the fibre $P(A)$ is a Heyting algebra, and for every arrow $\freccia{A}{f}{B}$ of $\mC$, $\freccia{P(B)}{P_f}{P(A)}$ is a morphism of Heyting algebras;

\item for any projection $\freccia{ A}{{\pr}}{B}$, the functor
\[ \freccia{P(B)}{{P_{\pr}}}{P(A)}\]
has a right adjoint $\forall_{\pr}$. 
\end{itemize}

\end{definition}
\begin{remark}\label{rem_BCC_forall}
    Notice that, for a first order hyperdoctrine, the right adjoints satisfy the Beck-Chevalley condition. Indeed, we can deduce the validity of this condition by the fact the left adjoints satisfy it.
    
    \end{remark}
\begin{example}[\cite{pitts02,hyland89}]\thlabel{ex:localic tripos}
Let $\mathsf{A}$ be a locale. The functor $\doctrine{\mathsf{Set}}{\mathsf{A}^{(-)}}$ assigning to a set $X$ the poset $\mathsf{A}^X$ of functions from $X$ to $\mathsf{A}$ with the point-wise order is a first order hyperdoctrine.
\end{example}

\begin{example}[\cite{pitts02,hyland89,Hofstra2006AllRI}]\thlabel{ex:realizability tripos}
Given a PCA $\pca{A}$ with elementary sub-PCA $\subpca{A}$, we can consider the relative realizability first order hyperdoctrine $\doctrine{\mathsf{Set}}{\mathcal{P}}$ 
over $\mathsf{Set}$: for each set $X$, the preorder $(\powerset(\pca{A})^X,\leq)$ is defined as the set of functions from $X$ to the powerset $\powerset(\pca{A})$ of $\pca{A}$ and, given two elements $\alpha$ and $\beta$ of $\powerset(\pca{A})^X$, we say that $\alpha\leq \beta$ if there exists an element $\ovln{a}\in \pca{A}'$ such that for all $x\in X$ and all $a\in \alpha (x)$, $\ovln{a}\cdot a$ is defined and it is an element of $\beta (x)$.  By standard properties of PCAs this relation is reflexive and transitive, i.e.\ it is a preorder.
Then $\mathcal{P} (X)$ is defined as the posetal reflection of $(\powerset(\pca{A})^X,\leq)$. 
The functor $\mathcal{P}$ acts by precomposition on arrows. 
\end{example}

\begin{example}[\cite{trottavalenti2025}]\thlabel{def:full medvedev doctrine}
    Given a PCA $\pca{A}$ with elementary sub-PCA $\subpca{A}$, the Medvedev doctrine $\doctrine{\set}{\MedvedevDoctrine}$ is defined as follows: for every set $X$ and every pair of functions $\varphi,\psi$ in $\powerset(\pca{A})^X$, we define 
    \begin{align*}
        \varphi \MedvedevOrder \psi :\iff & (\exists \ovln{a}\in \subpca{A})(\forall x \in X)(\forall b\in \psi(x))(\exists a \in \varphi(x))(\ovln{a}\cdot b = a)\\
            \iff & (\exists \ovln{a}\in \subpca{A})(\forall x \in X)(\ovln{a}\cdot\psi(x) \subseteq \varphi(x)).
    \end{align*}
 This preorder induces a partial order (by taking the ordinary poset reflection) on functions in $\powerset(\pca{A})^X$. The doctrine is then defined by precomposition on arrows. The name of this doctrine is motivated by the fact that $\MedvedevDoctrine (1)$ is a generalization of the notion of Medvedev degrees to PCAs. 
 This doctrine is existential, but it is not a first order hyperdoctrine as it lacks the implication in the fibres. Indeed, it is a co-Heyting algebra i.e.\ a Brouwer algebra, see \cite[Thm.\ 9.1]{Sorbi1996}. This doctrine is the dual, i.e.\ obtained by inverting the order in the fibres, of the relative realizability doctrine presented in \thref{ex:realizability tripos}, see \cite{trottavalenti2025}.
\end{example}

\begin{example}[\cite{maiettirosolinipasquali2019}]\thlabel{ex:weak_sub_partition_asm}
    The weak subobject doctrine $\doctrine{\partAsm (\pca{A})}{\Psi}$  on the category of partitioned assemblies is a first order hyperdoctrine. This is due to the fact that $\partAsm (\pca{A})$ has weak dependent products. See, e.g. \cite{menni2002}.
\end{example}

\subsection{Triposes}


The notion of \emph{tripos} was originally introduced in \cite{hyland89} by Hyland, Johnstone and Pitts, motivated by the desire to formally explain in what sense Higgs' description of localic toposes as \emph{A-valued sets}~\cite{higgs84} and Hyland's realizability toposes \cite{hyland1982effective} are instances of the same construction.

Drawing up on Lawvere's treatment of logic in terms of hyperdoctrines, a tripos can be presented as a first order hyperdoctrine equipped with \emph{weak power objects}.

\begin{definition}[weak predicate classifier]
A primary doctrine $\doctrine{\mC}{P}$ has a
\textbf{weak predicate classifier} (or generic element) if there exists an object $\Omega$ of $\mC$ together with an element $\sigma$ of $P(\Omega)$ such that for every object $A$ of $\mC$ and every $\alpha$ of $P(A)$ there exists a morphism $\freccia{A}{\{\alpha\}}{\Omega}$ such that $P_{ \{\alpha\}}(\sigma)=\alpha$.
\end{definition}
\begin{definition}[weak power objects]
A primary doctrine $\doctrine{\mC}{P}$ has
\textbf{weak power objects} if for every object $X$ of $\mC$ there exists an object $\mathrm{P}X$ and an element $\in_X$ of $P(X\times\mathrm{P}X)$ such that for every $\beta$ of $P(X\times Y)$ there exists an arrow $\freccia{Y}{\{\beta\}_X}{\mathrm{P}X}$ such that $\beta=P_{\id_X\times \{\beta\}_X}(\in_X)$.

\end{definition}
\begin{remark}\thlabel{rem:weak cart closed power ob iff pred clas}
Notice that if $\doctrine{\mC}{P}$ has weak power objects, then it has a weak predicate classifier given by $\mathrm{P}1$ and $\in_1$.
The converse holds when the base category is weakly cartesian closed \cite{hyland89}. In particular, if $\doctrine{\mC}{P}$ has a weak predicate classifier and the base category is weakly cartesian closed then, for every object $A$ of $\mC$, we can define an object $\mathrm{P}A:=\Omega^A$ and an element $\in_A:=P_{\mathsf{ev}}(\sigma)$ of $P(A\times \mathrm{P}A)$
where $\freccia{A\times \Omega^A}{\mathsf{ev}}{\Omega}$ is the evaluation arrow, and these assignments give to $P$ the structure of weak power objects.

\end{remark}
\begin{definition}[tripos]\label{def:tripos}
A first order hyperdoctrine $\doctrine{\mC}{P}$ is called a \textbf{tripos} if it has  weak power objects.
\end{definition}

\begin{example}
    The localic hyperdoctrine presented in \thref{ex:localic tripos} is a tripos,  In particular, the weak predicate classifier is the object $\Omega:=\mathsf{A}$  together with $\sigma:=\id_{\mathsf{A}}$.
\end{example}

\begin{example}[\cite{pitts02,hyland89}]
    The realizability hyperdoctrine presented in \thref{ex:realizability tripos} is a tripos. In particular, the weak predicate classifier is the object $\Omega:=\powerset(\pca{A})$  together with $\sigma:=\id_{\powerset(\pca{A})}$.
\end{example}
\begin{example}
 The Medvedev doctrine $\doctrine{\set}{\MedvedevDoctrine}$ presented in \thref{def:full medvedev doctrine} has a weak predicate classifier (and hence, weak power objects by \thref{rem:weak cart closed power ob iff pred clas}) given by the object $\Omega:=\powerset(\pca{A})$  together with $\sigma:=\id_{\powerset(\pca{A})}$.
\end{example}

\begin{example}\thlabel{ex:weak_sub_partition_asm_tripos}
    The weak subobject doctrine $\doctrine{\partAsm (\pca{A})}{\Psi}$  on the category of partitioned assemblies is a tripos. The existence of weak power objects follows from the fact that $\partAsm (\pca{A})$ is weakly cartesian closed and it has a generic proof~\cite{menni2002}. Indeed, the generic proof corresponds precisely to the notion of weak predicate classifier of $\doctrine{\partAsm (\pca{A})}{\Psi}$, and hence, by \thref{rem:weak cart closed power ob iff pred clas}, we have the existence of weak power objects for this doctrine.
\end{example}

\subsection{The full existential completion}\thlabel{sec_exist_comp}
Over the years, several authors observed that many triposes coming from the realizability world are instances of a free construction that freely adds left adjoints along all the morphisms of the base category of a doctrine, see e.g. the work by Hostra~\cite{Hofstra2006AllRI},  Frey~\cite{Frey2014AFS,Frey2020} and the work by Maietti and the second author~\cite{trottamaietti2020}.

In the recent work~\cite{trottavalenti2025} it turns out that also many doctrines coming from the computability world are instances of these kind of completions that freely add left and/or right adjoints to a given doctrine. 

In this section, we quickly recall the free construction called full existential completion in \cite{trottamaietti2020} and $\exists$-completion in \cite{Frey2020} that adds left adjoints to a doctrine, along all the morphisms of the base. This is a particular case of the most general construction presented in \cite{trotta2020}, which freely adds left adjoints along a given class of arrows (closed under pullbacks, compositions and identities) to a given doctrine.

We start by quickly recalling the full existential completion from \cite{trotta2020,trottamaietti2020}.

\noindent
\textbf{Full existential completion}. 
Let $\doctrine{\mC}{P}$ be a primary doctrine where $\mC$ is a category with all finite limits.  For every object $A$ of $\mC$ consider the following preorder:
\begin{itemize}
\item \textbf{objects:} pairs $(\arrowup{B}{f}{A},\alpha)$, where $\arrow{B}{f}{A}$ is an arrow of $\mC$ and $\alpha\in P(B)$.
\item \textbf{order:}  $(\arrowup{B}{f}{A},\alpha)\leq (\arrowup{C}{g}{A},\beta)$ if there exists an arrow $\arrow{B}{h}{C}$ of $\mC$ such that the diagram 
\[\begin{tikzcd}
	& B \\
	C & A
	\arrow["h"', dashed, from=1-2, to=2-1]
	\arrow["f", from=1-2, to=2-2]
	\arrow["g"', from=2-1, to=2-2]
\end{tikzcd}\]
commutes and
\[ \alpha\leq P_{h}(\beta).\]

\end{itemize}
We denote by
$P^{\exists}(A)$ the partial order obtained by identifying two
objects when
$$(\arrowup{B}{h}{A}, \alpha)\gtreqless (\arrowup{D}{f}{A}, \gamma)$$
in the usual way. With a small abuse of notation, we identify an equivalence class with one of its representatives.

Given a morphism $\arrow{A}{f}{B}$ in $\mC$, let $P^{\exists}_f(\arrowup{C}{g}{B},\beta)$ be the object 
\[(\arrowup{D}{\pbmorph{f}{g}}{A},\; P_{g^*f}(\beta) )\]
where $f^*g$ and $g^*f$ are defined by the pullback
\[\xymatrix@+1pc{
D\ar[d]_{f^*g} \pullbackcorner \ar[r]^{g^*f} &C\ar[d]^g\\
A \ar[r]_f & B.
}\]
\begin{theorem}\label{exists}\thlabel{thm: compex P is existential and RC}
The doctrine $\doctrine{\mC}{P^{\exists}}$ is an existential doctrine.
\end{theorem}

The previous construction extends to a 2-adjunction from the 2-category of primary doctrines to the 2-category of existential doctrines. We refer to \cite{trotta2020} for a complete analysis of this construction.

\begin{theorem}\thlabel{thm:full existential comp}
The assignment $P\mapsto P^{\exists}
$ extends to a 2-functor 
	\[\begin{tikzcd}
		\PD && \ED
		\arrow[""{name=0, anchor=center, inner sep=0}, "(-)^{\exists}",curve={height=-14pt}, from=1-1, to=1-3]
		\arrow[""{name=1, anchor=center, inner sep=0}, " ",curve={height=-14pt}, hook',from=1-3, to=1-1]
		\arrow["\dashv"{anchor=center, rotate=-90}, draw=none, from=1, to=0]
	\end{tikzcd}\]
	from the 2-category $\PD$ of primary doctrines to the 2-category $\ED$ of existential doctrines, and it is left adjoint to the inclusion functor.
    
\end{theorem}
\begin{remark}\thlabel{rem:universal completion}
    It is worth to observe that, combining the existential completion with the functor $\freccia{\mathsf{Pos}}{(-)^{\op}}{\mathsf{Pos}}$ inverting the order of a poset, one can define also the \emph{universal completion} of a functor $P\colon\mC^{\op}\to\mathsf{Pos}$ as $P^{\forall}:=((P^{\op})^{\exists})^{\op}$, see \cite{trotta23TCS,trotta-lfcs2022}.
\end{remark}
\begin{example}
    Realizability doctrines are relevant examples of doctrines arising as full existential completions. The original observation of this result is due to Hofstra \cite{Hofstra2006AllRI}, while a more general analysis of doctrines arising as full existential completions can be found in \cite{trottamaietti2020,Frey2020}.
\end{example}
\begin{example}[\cite{trottamaietti2020}]\thlabel{ex_weaksub_is_full_ex_comp}
    Every weak subobject doctrine $\doctrine{\mC}{\Psi}$ is an instance of the full existential completion. In particular, $\Psi$ is the existential completion $\mathsf{t}^{\exists}$ of the trivial doctrine $\doctrine{\mC}{\mathsf{t}}$ sending every object of $\mC$ into the poset $\mathsf{t}(X):=\{\bullet\}$ with just one element.
\end{example}


\section{Instance reducibility and extended Weihrauch degrees}\label{sec:extendedWeihrauch}
Recently, Bauer \cite{Bauer2021} introduced another generalization of Weihrauch reducibility, called \emph{extended Weihrauch reducibility}, that can be seen as another way to generalize Weihrauch reducibility to multi-represented spaces. Moreover, he proved that this notion is equivalent to a second notion, called  \emph{instance reducibility},  which commonly appears in reverse constructive mathematics.

We start by recalling the main definitions from \cite{Bauer2021}:

\begin{definition}[Extended Weihrauch reducibility {\cite[Def.\ 3.7]{Bauer2021}}]
    \thlabel{def:extwei_relation}
    Let $\pca{A}$ be a  PCA, and $\pca{A}'$ be an elementary sub-PCA of $\pca{A}$. An \textbf{extended Weihrauch predicate} is a function $f\function{\pca{A}}{\powerset\powerset(\pca{A})}$.

    Given two extended Weihrauch predicates $f,g$, we say that $f$ is \textbf{extended-Weihrauch reducible} to $g$, and write $f \extweireducible g$ if there are $\ell_1, \ell_2\in \subpca{A}$ such that for every $p\in \pca{A}$ such that $f(p)\neq \emptyset$
\begin{itemize}
    \item $\ell_1\cdot p\downarrow$ and $g(\ell_1\cdot p)\neq \emptyset$;
    \item for every $A \in f(p)$ there is $B \in g(\ell_1\cdot p)$ such that for every $q \in B$, $\ell_2\cdot \pairing{p,q} \downarrow$ and $\ell_2\cdot \pairing{p,q} \in A$.
\end{itemize}
\end{definition}

Following the notation used in \cite{Bauer2021}, we denote by $\mathsf{Pred}(X):=\powerset (\pca{A})^{|X|}$ the set of \textbf{realizability predicates} of an assembly $X$. 
\begin{definition}[Instance reducibility]\thlabel{def:instance_reducibility_asm}
    A realizability predicate $\phi \in \mathsf{Pred}(X)$ is instance reducible to $\psi \in \mathsf{Pred}(Y) $ when there exist $l_1,l_2\in \pca{A}'$ such that for every $s\in \pca{A}$ and $x\in |X|$, if $s\name_X x$ then there exists an element $y\in |Y|$ such that $l_1\cdot s\name_Y y$ and for every $p\in \psi (y)$, $l_2\cdot \pairing{s, p}\in \phi (x)$.
\end{definition}
\begin{lemma}\thlabel{lemma:  instance red}
    Every realizability predicate $\phi \in \mathsf{Pred}(X)$ of an assembly $X$ is equivalent (w.r.t. instance reducibility) to a realizability predicate  $\alpha_{\phi} \in \mathsf{Pred}(X_{\phi})$
    of a partitioned assembly $X_{\phi}$. In particular, $|X_{\phi}|:=\{(x,s)| x\in |X|, s\name_X x\} $, $s\name_{X_{\phi}} (x,s)$, and $\alpha_{\phi}(x,s):=\{\pairing{s,q}|q\in \phi(x)\}$ works.
\end{lemma}
\begin{proof}
    We first show that $\phi$ is instance reducible to $\alpha_{\phi}$. We claim that $l_1$ is the identity and $l_2=\lambda x.\sprPCA(\sprPCA x)$. Let us consider $x\in |X|$ and $s\name_X x$. Then we choose $(x,s)\in X_{\phi}$ and, by definition, $s\name_{X_{\phi}} (x,s)$. Now, for every $p\in \alpha_{\phi}(x,s)$, i.e.\ for every $\pairing{s,q}$ with $q\in \phi(x)$, we have that $l_2\cdot \pairing{s,\pairing{s,q}}\in \phi (x)$. So, we can conclude that $\phi$ is instance reducible to $\alpha_{\phi}$.

    Now we prove that $\alpha_{\phi}$ is instance reducible to $\phi$. We claim that $l_1$ the identity and $l_2$ the identity provide us the reducibility. For $(x,s)\in |X_{\phi}|$, and $s\name_{X_{\phi}} (x,s)$, we chose $x\in |X|$ and, by definition $s\name_X x$. Now, for every $q\in \phi(x)$, $\pairing{s,q}\in \alpha_{\phi}(x,s) $ by definition. This concludes the proof.
\end{proof}

Therefore, the poset reflection of realizability predicates over assemblies is equivalent to that of realizability predicates over partitioned assemblies. 
\begin{remark}\thlabel{rem_instance_reduced_to_parAsm}
    
Notice that in the case of a realizability predicate over partitioned assemblies, \thref{def:instance_reducibility_asm} can be presented in the following way. 

Let $X,Y$ be partitioned assemblies. A realizability predicate $\phi \in \mathsf{Pred}(X)$ is instance reducible to $\psi \in \mathsf{Pred}(Y) $ when there exist a morphism of partitioned assemblies $H\function{X}{Y}$ and $l\in \pca{A}'$ such that for every $x\in |X|$ and $s\name_X x$, we have that for every $p\in \psi (H(x))$, $l\cdot \pairing{s, p}\in \phi (x)$.
\end{remark}

We recall that Weihrauch reductions and instance reductions can be proved to be equivalent preorders (see \cite{Bauer2021}). 

\begin{proposition}\thlabel{prop_ext_weih_equiv_instance_reduction}
   Weihrauch reductions and instance reductions in $\partAsm(\pca{A},\subpca{A})$ are equivalent preorders.
\end{proposition}

\section{A tripos for instance reducibility}
Now we explain how the notion of instance reducibility can be presented in the language of doctrines.

\begin{definition}[elementary instance reducibility doctrine]
    We define the \textbf{elementary instance reducibility doctrine} $\doctrine{\partAsm(\pca{A},\subpca{A})}{\EWeiElementary}$ as follows: for every object $X=(|X|,\name_X)$ of $\partAsm(\pca{A},\subpca{A})$, the objects of $\EWeiElementary(X)$ are functions $f\function{|X|}{\powerset(\pca{A})}$. For every pair of maps $f,g$ in $\powerset(\pca{A})^{|X|}$, we define $f \EWeiOrderElementary g$ iff there is $\ovln{h}\in \subpca{A}$ such that for $p\name_X x$
    \[ (\forall q \in g(x))(\ovln{h}\cdot \pairing{p,q}\in f(x)).\]
    As usual, $\EWeiElementary(X)$ is the posetal reflection of $(\powerset(\pca{A})^{|X|},\EWeiOrderElementary)$, and the action of $\EWeiElementary$ on the morphisms of $\partAsm(\pca{A},\subpca{A})$ is defined by the (suitable) pre-composition.
\end{definition}
\begin{remark}
    Notice that the inf-semilattice structure of the fibres of the doctrine $\doctrine{\partAsm(\pca{A},\subpca{A})}{\EWeiElementary}$ is given as follows: the top element of the fibre $\EWeiElementary(X)$  is the function $\top\function{|X|}{\powerset(\pca{A})}$ assigning to every element the empty set. The binary inf $f\wedge g\function{|X|}{\powerset(\pca{A})}$ is given by the function $(f\wedge g)(x):= f(x)\oplus g(x)$. Moreover, since the action of $\EWeiElementary$ on the morphisms of $\partAsm(\pca{A},\subpca{A})$ is defined by the pre-composition, the inf-semilattice structure of the fibres is preserved by the reindexings.
\end{remark}
\begin{proposition}\thlabel{prop:eiR is universal com}
    The doctrine $\doctrine{\partAsm(\pca{A},\subpca{A})}{\EWeiElementary}$ is isomorphic to the doctrine $\doctrine{\partAsm(\pca{A},\subpca{A})}{(-)^{\op}\circ \Psi}$ obtained by composing $\Psi$ with the functor inverting the order of a poset. Therefore, $\EWeiElementary$ is the universal completion of the trivial doctrine over $\partAsm(\pca{A},\subpca{A})$.
\end{proposition}
\begin{proof}
 Given a partitioned assembly $X=(|X|,\name_X)$ we define two morphisms of posets $h_1\function{\EWeiElementary(X)}{\Psi(X)^{\op}}$ and $h_2\function{\Psi(X)^{\op}}{\EWeiElementary(X)}$ as follows: given an element $f\function{|X|}{\powerset(\pca{A})}$, we define $h_1(f)$ as the first projection from the partitioned assembly $X_f$, where $|X_f|:=\{(x,s)| x\in |X|, s\in f(x)\}$ and $\pairing{p,s} \name_{X_f} (x,s)$ where $p\name_X x$.  It is straightforward to check that this assignment is monotone.

 Now let us consider an element $g\function{(|Y|,\name_Y)}{(|X|,\name_X)}$ of $\Psi(X)^{\op}$. We define $h_2(g)$ as the function $h_2(g)\function{|X|}{\powerset(\pca{A})}$ mapping $$x\mapsto\{ a\in \pca{A}| \exists y\in g^{-1}(x), a\name_Y y\}.$$ Again, it is straightforward to check that this assignment is monotone.
 
 Now it is direct to check that $h_2\circ h_1$ is the identity, because $ h_2 h_1 (f) (x)=\{\pairing{p,s} | s \in f(x), p\name_X x\}$, so one can easily check that both $h_2 h_1 (f)\EWeiOrderElementary f $ and $ f\EWeiOrderElementary h_2 h_1 (f) $ hold.

 To prove that $h_1\circ h_2$  is the identity we need to check that, given a morphism $g\function{(|Y|,\name_Y)}{(|X|,\name_X)}$ of $\Psi(X)^{\op}$, there exists two morphism of paritioned assemblies such that 
\[\begin{tikzcd}
	(|Y|,\name_Y) && (|X|,\name_X) \\
	& \qquad \;(|X_{h_2(g)}|,\name_{X_{h_2(g)}})
	\arrow["g", from=1-1, to=1-3]
	\arrow["{a_1}", from=1-1, to=2-2]
	\arrow["{a_2}", shift left=3, from=2-2, to=1-1]
	\arrow["{\pi_1}"', from=2-2, to=1-3]
\end{tikzcd}\]
Notice that $|X_{h_2(g)}|= \{(x,s)\in X\times \pca{A}|  s\in h_2(g)(x)\}$, and by definition of $h_2$,  $|X_{h_2(g)}|= \{(x,s)\in X\times \pca{A}| \exists y\in g^{-1}(x) (s\name_Y y)\}$. Now, $a_1$ is defined by $y\mapsto (g(y), s)$ where $s\name_Y y$, and this it is morphism of partitioned assemblies because is $g$ is so. 

Now, we define $a_2$ by using the Axiom of Choice as follows: $a_2$ maps $ (x,s)$ of $|X_{h_2(g)}|$ to an element of $g^{-1}(x)$ such that $s\name_Y g(x)$. One can easily check that $a_2$ is a morphism of partitioned of assemblies (it is tracked by the second
projection).

With these two morphisms we have that the two triangles of the previous diagram commute, hence $h_1 h_2(g)\leq g$ and $g\leq h_1 h_2(g)$ in $\Psi(X)^{\op}$.

 Finally, combining \thref{rem:universal completion} with \thref{ex_weaksub_is_full_ex_comp} we obtain that $\EWeiElementary$ is the universal completion of the trivial doctrine $\partAsm(\pca{A},\subpca{A})$.
\end{proof}

\textbf{Notation: }For the sake of convenience, from now on we change our notation of partitioned assemblies. Instead of denoting an arbitrary partitioned assembly as a pair $X=(|X|,\name_X)$, we will denote it as a pair $(X,\varphi)$ where $X$ is a set and $\varphi$ is a function from $X$ to $\mathbb{A}$. Indeed, the realizability relation $\name_X$ of a partitioned assembly $X$ corresponds to the function sending each $x\in |X|$ to the unique $a\in \mathbb{A}$ such that $a\name_X x$.

By \thref{rem_instance_reduced_to_parAsm}  we can easily prove that the poset $\EWeiElementary^{\exists}(1)$ given by the existential completion (see Section~\ref{sec_exist_comp}) of $\EWeiElementary$ is equivalent to that of instance reducibilities. 
Motivated by the this fact, we introduce the following definition:
\begin{definition}[doctrine of instance reducibilities]
    We call the doctrine $\doctrine{{\partAsm(\pca{A},\subpca{A})}}{\InstRedDoc:=\EWeiElementary^{\exists}}$ the \textbf{doctrine of instance reducibilities}.
\end{definition}
 
We start by providing an explicit description of this doctrine
$$\doctrine{\partAsm(\pca{A},\subpca{A})}{\InstRedDoc}$$

\begin{enumerate}
\item the elements of $\InstRedDoc(X,\varphi)$ are pairs $((Y,\psi)\xrightarrow{f} (X,\varphi),\alpha)$ where $f$ is an arrow in $\partAsm(\pca{A},\subpca{A})$ and $\alpha\function{Y}{\powerset(\mathbb{A})}$ is a function;
\item $((Y,\psi)\xrightarrow{f} (X,\varphi),\alpha)\leq((Z,\eta)\xrightarrow{g} (X,\varphi),\beta)$ in $\InstRedDoc(X,\varphi)$ if there exists $h:(Y,\psi)\rightarrow (Z,\eta)$ in $\partAsm(\pca{A},\subpca{A})$ such that $f=g\circ h$ and there exists  $\ell\in \mathbb{A}'$ such that 
$$\forall y\in Y\forall q\in \beta(h(y))\,(\ell\cdot\langle \psi(y),q\rangle\in \alpha(y));$$
\item pairs which are equivalent with respect to $\leq$ are identified;
\item if $h\function{(X',\varphi')}{(X,\varphi)}$ is an arrow of $\partAsm(\pca{A},\subpca{A})$, then $$\InstRedDoc_{h}(((Y,\psi)\xrightarrow{f} (X,\varphi),\alpha))$$ is given by a pair
$$((P,\pi)\xrightarrow{h^{*}(f)}(X',\varphi'),\alpha\circ f^{*}(h))$$
where the following diagram is a pullback in $\partAsm(\pca{A},\subpca{A})$:
$$\xymatrix{
(P,\pi)\ar[r]^-{f^{*}(h)}\ar[d]_-{h^{*}(f)} &(Y,\psi)\ar[d]^-{f}\\
(X',\varphi')\ar[r]_-{h}    &(X,\varphi)\\
}$$
\end{enumerate}
It may be helpful to complement the formal definition of the doctrine of instance reducibilities with a more intuitive description: given an element $y \in Y$, one can think of it as carrying two distinct kinds of realizers. On the one hand, $\psi(y)$, coming from the partial assembly structure, can be interpreted as evidence that $y$ is realized (or ``true''). On the other hand, the elements of $\alpha(y)$ can be understood as providing evidence that $y$ is not realized (or ``false'').

From this perspective, the structure captures a form of computational or logical tension between positive and negative information associated to each instance, which is reflected in the definition via existential completion.

Combining \thref{rem_instance_reduced_to_parAsm} with the second point in the previous description we obtain the following result:
\begin{theorem}\thlabel{thm_eW_1_is_Wei}
 Instance reducibilities and  $\InstRedDoc(1)$ are equivalent posets.
\end{theorem}

Now we start investigating in detail the structure of the fibres of $\InstRedDoc$:
\begin{proposition}For every partitioned assembly $(X,\varphi)$, the poset $\InstRedDoc(X,\varphi)$ is a bounded lattice.
\end{proposition}
\begin{proof}

\begin{enumerate}
\item The existence of the maximum and of binary infima is guaranteed by \thref{thm: compex P is existential and RC}. However, let us recall their construction for later use.
 
 The maximum is given by $((X,\varphi)\xrightarrow{\id_X} (X,\varphi),x\mapsto \emptyset)$. 

If $(f,\alpha)$ and $(g,\beta)$ are in $\InstRedDoc(X,\varphi)$ with $f$ and $g$ having domain $(Y,\psi)$ and $(Z,\eta)$, respectively, their infimum $(f,\alpha)\wedge (g,\beta)$ is given by 
$$(f\circ \pi_1,\alpha\circ \pi_1\oplus \beta\circ \pi_2)$$
where
\[\begin{tikzcd}
	{(P,\rho)} & {(Z,\eta)} \\
	{(Y,\psi)} & {(X,\varphi)}
	\arrow["{\pi_2}", from=1-1, to=1-2]
	\arrow["{\pi_1}"', from=1-1, to=2-1]
	\arrow["\lrcorner"{anchor=center, pos=0.125}, draw=none, from=1-1, to=2-2]
	\arrow["g", from=1-2, to=2-2]
	\arrow["f"', from=2-1, to=2-2]
\end{tikzcd}\]
is the usual pullback diagram in $\partAsm(\pca{A},\subpca{A})$, where $$P=\{(y,z)\in Y\times Z|\,f(y)=g(z)\}$$ and $\rho(y,z):= \pairing{\psi(y),\eta(z)}$, and $(\alpha\circ \pi_1\oplus \beta\circ \pi_2)(y,z):=\alpha(y)\oplus \beta(z)$.
\item[]
\item The minimum is given by $((\emptyset,!)\xrightarrow{!_X} (X,\varphi),!)$. Indeed, for every $(f,\alpha)$ in $\InstRedDoc(X,\varphi)$, one has $!_X=f\circ !_Y$ and the other condition is vacuously satisfied by any $\ell\in \mathbb{A}'$ since we universally quantify over $\emptyset$.  
\item[]
\item A supremum for $(f,\alpha)$ and $(g,\beta)$ having domain $(Y,\psi)$ and $(Z,\eta)$, respectively, is given by 
$$(f,\alpha)\vee (g,\beta):=([f,g],[\alpha,\beta])$$ where $C:=Y+Z$  and $[f,g]\function{(C,\chi)}{(X,\varphi)}$ is defined as usual for coproducts of sets, and $\chi\function{C}{\pca{A}}$ is defined by
\[
\begin{cases}
  \chi(0,y) =\pairing{\kPCA, \psi(y)}  \\
  \chi(1,z) = \pairing{\bkPCA, \eta(z)} 
\end{cases}
\]
Similarly, $[\alpha,\beta]\function{C}{\powerset (\pca{A})}$ is defined by 
\[
\begin{cases}
[\alpha,\beta](0,y)= \alpha(y) \\
[\alpha,\beta](1,z)= \beta(z)
\end{cases}
\]
If $(f,\alpha)\leq (h,\gamma)$ and $(g,\beta)\leq (h,\gamma)$ in $\InstRedDoc(X,\varphi)$ and the domain of $h$ is $(E,\epsilon)$, then there exist two arrows of partitioned assemblies $k_1:(Y,\psi)\rightarrow (E,\epsilon)$ and $k_2:(Z,\eta)\rightarrow (E,\epsilon)$, and $\ell_1,\ell_2\in \mathbb{A}'$ such that:
\begin{enumerate}
\item $f=h\circ k_1$;
\item $g=h\circ k_2$;
\item $\forall y\in Y\,\forall q\in \gamma(k_1(y))\,(\ell_1\cdot \langle\psi(y),q\rangle\in \alpha(y))$
\item $\forall z\in Z\,\forall r\in \gamma(k_2(z))\,(\ell_2\cdot \langle\eta(z),r\rangle\in \beta(z))$
\end{enumerate}

If we consider the arrow given by $[k_1,k_2]:(C,\chi)\rightarrow (E,\epsilon)$, then $[f,g]=h\circ [k_1,k_2]$. In order to show that $(f,\alpha)\vee (g,\beta)\leq (h,\gamma)$ we have to define an element $\ell\in \pca{A}'$ such that
$$\forall y\in Y\,\forall q\in \gamma(k_1(y))\,\ell\cdot \langle \langle \kPCA,\psi(y)\rangle, q\rangle\in  \alpha(y)$$
$$\forall z\in Z\,\forall r\in \gamma(k_2(z))\,\ell\cdot \langle \langle \bkPCA,\eta(z)\rangle, r\rangle\in  \beta(z)$$
One can just take $\ell$ to be

$$\lambda \xi.(\fprPCA(\fprPCA\xi))(\ell_1\langle \sprPCA(\fprPCA\xi),\sprPCA\xi \rangle )(\ell_2\langle \sprPCA(\fprPCA\xi),\sprPCA\xi \rangle )$$


To conclude $(f,\alpha)\leq (f,\alpha)\vee (g,\beta)$ via $i_1$ and $\ell_1:=\sprPCA$ while $(g,\beta)\leq (f,\alpha)\vee (g,\beta)$ via $i_2$ and $\ell_2:=\sprPCA$.

\item[]
\end{enumerate}
\end{proof}
\begin{proposition}\label{heytobj}For every partitioned assembly $(X,\varphi)$, the poset $\InstRedDoc(X,\varphi)$ is a Heyting algebra.
\end{proposition}
\begin{proof}

 A Heyting implication $(f,\alpha)\Rightarrow(g,\beta)$ from $(f,\alpha)$ to $(g,\beta)$ is given by the first projection $\pi_1\function{((Y,\psi)\rhd_{(X,\varphi)} (Z,\eta), \psi \rhd_{(X,\varphi)} \eta)}{(X,\varphi) }$ from the partitioned assembly defined as follows:

\begin{itemize}
    \item the underlying set $(Y,\psi)\rhd_{(X,\varphi)} (Z,\eta) $ is given by
$$(\{(x,k,R,r,l)|\;x\in X,r\in \mathbb{A},\qquad\qquad\qquad\qquad\qquad\qquad\qquad\qquad $$
$$ r\Vdash k:(f^{-1}(x),\psi)\rightarrow(g^{-1}(x),\eta)\text{ in }\partAsm(\pca{A},\pca{A}),R\subseteq \mathbb{A},$$
$$l\in \mathbb{A},\forall y\in f^{-1}(x)\,\forall q\in \beta(k(y))\,(l\cdot\langle \psi(y),q\rangle \in R\oplus\alpha(y))\}$$
\item the function $\psi \rhd_{(X,\varphi)} \eta$ is given by the assignment 
$$(x,k,R,r,l)\mapsto \langle \varphi(x),\langle r,l\rangle\rangle)$$
\end{itemize}

together with the predicate defined by the third projection $$\pi_3\function{(Y,\psi)\rhd_{(X,\varphi)} (Z,\eta)}{\powerset (\pca{A})}$$ 
$$(x,k,R,r,l)\mapsto R$$

Assume now $(h,\gamma)\wedge (f,\alpha)\leq (g,\beta)$, where $h:(E,\epsilon)\rightarrow (X,\varphi)$. 
Then, there exists a morphism of partitioned assemblies $m\function{(P,\rho)}{(Z,\eta)}$ from the pullback of $h$ and $f$,
$$P=\{(e,y)\in E\times Y|\,h(e)=f(y)\}$$
with $\rho (e,y)= \pairing{\epsilon(e), \psi(y)}$, such that the triangle
\[\begin{tikzcd}
	& {(P,\rho)} \\
	{(Z,\eta)} & {(X,\varphi)}
	\arrow["m"', from=1-2, to=2-1]
	\arrow["{f\circ (f^*h)}", from=1-2, to=2-2]
	\arrow["g"', from=2-1, to=2-2]
\end{tikzcd}\]
commutes
i.e.\ $(g\circ m)(e,y)=f(y)=h(e)$, and an element $\ell\in \mathbb{A}'$ such that 
$$\forall (e,y)\in P\,\forall q\in \beta(m(e,y))\,(\ell\cdot\pairing{\rho (e,y),q}\in \gamma(e)\oplus \alpha(y))$$

We must show that $(h,\gamma)\leq (f,\alpha)\Rightarrow (g,\beta)$. First, we have to define a morphism of partitioned assemblies 
$$ \bar{m}\function{(E,\epsilon)}{((Y,\psi)\rhd_{(X,\varphi)} (Z,\eta), \psi \rhd_{(X,\varphi)} \eta) }$$

We then consider the function $\bar{m}$ which sends every $e\in E$ into $$(h(e),y\mapsto m(e,y),\gamma(e),\lambda \xi.r\cdot \pairing{\epsilon(e),\xi},\lambda \xi.\ell\cdot \langle \pairing{ \epsilon(e),\fprPCA(\xi)}, \sprPCA(\xi)\rangle)$$
where $r$ is a chosen realizer for the morphism $m$ of partitioned assemblies. Notice that the function $y\mapsto m(e,y)$ is a well-defined morphism in $\partAsm(\pca{A},\pca{A})$ from $(f^{-1}(h(e)),\psi)$ to $(g^{-1}(h(e)),\eta)$, 
because, for every $y\in f^{-1}(h(e)) $, by definition of $P$, we have that $(e,y)\in P$, and hence $m(e,y)\in g^{-1}(h(e))$ (because $(g\circ m)(e,y)=h(e)$), and moreover it is direct to check
that $\lambda \xi.r\cdot \langle \epsilon(e),\xi\rangle$ is a realizer for such a function to be a morphism in $\partAsm(\pca{A},\pca{A})$ using the fact that $r$ is a realizer of $m\function{(P,\rho)}{(Z,\eta)}$, i.e.\ that $r\cdot \pairing{\epsilon (e),\psi(y)}=\eta(m(e,y))$ for every $(e,y)\in P$.
Moreover, it is immediate to check that the fifth component of $\overline{m}(e)$ satisfies the required property, as a direct consequence of the property of $\ell$ shown above.

It is also straightforward to check that the function $\bar{m}$ provides a morphism $ \bar{m}\function{(E,\epsilon)}{((Y,\psi)\rhd_{(X,\varphi)} (Z,\eta), \psi \rhd_{(X,\varphi)} \eta) }$ of partitioned assemblies and that, by definition, it makes the desired triangle commute, i.e.\ $\pi_1\circ \bar{m}=h$. Then, we need to find an element $\ell'\in \mathbb{A}'$ such that
$$\forall e\in E\,\forall q\in  (\pi_3\circ \bar{m})(e )\,(\ell'\cdot \pairing{\epsilon(e),q} \in \gamma(e)).$$
But this is quite easy as, by definition of $\bar{m}$, we have that $(\pi_3\circ \bar{m})(e)=\gamma (e)$, so we can put $\ell'=\sprPCA$.\\

Conversely, let us suppose that  

$$(h,\gamma)\leq (f,\alpha)\Rightarrow (g,\beta)$$
via a morphism of partitioned assemblies 
$$n\function{(E,\epsilon)}{((Y,\psi)\rhd_{(X,\varphi)} (Z,\eta), \psi \rhd_{(X,\varphi)} \eta)}$$
$$e\mapsto (h(e),k(e),R(e),r(e),l(e))$$
(realized by $\overline{r}\in \pca{A}'$) 

and an element $\ell\in \pca{A}'$ such that
\begin{equation}\label{eq_sec_verso_implica}
    \forall e\in E\,\forall q\in  R(e )\,(\ell\cdot \pairing{\epsilon(e),q} \in \gamma(e)).
\end{equation}

Now we need to show that $(h,\gamma)\wedge (f,\alpha)\leq (g,\beta)$. First, we can define an arrow $\widetilde{n}\function{(P,\rho)}{(Z,\eta)}$ from the pullback $(P,\rho)$ of $h$ and $f$ to $(Z,\eta)$ as follows:
$$(e,y)\mapsto k(e)(y)$$
This is a morphism of partitioned assemblies, since if $e\in E$, $y\in Y$ and $h(e)=f(y)$, then $(\lambda \xi.(\fprPCA(\sprPCA(\overline{r}(\fprPCA\xi))))(\sprPCA\xi))\langle\epsilon(e),\psi(y)\rangle=\eta(k(e)(y))$. 

Moreover, we have that the triangle
\[\begin{tikzcd}
	& {(P,\rho)} \\
	{(Z,\eta)} & {(X,\varphi)}
	\arrow["\widetilde{n}"', from=1-2, to=2-1]
	\arrow["{h\circ(h^*f)}", from=1-2, to=2-2]
	\arrow["g"', from=2-1, to=2-2]
\end{tikzcd}\]
commutes because for every $(e,y)\in P$ we have that, by assumption, $\widetilde{n}(e,y)\in g^{-1}(h(e))$ (because $h=\pi_1 \circ n$), i.e.\ $(g\circ \widetilde{n})(e,y)= h(e)=(h\circ(h^*f))(e,y)$.

Now, we need to find an  element $\ell'\in \mathbb{A}'$ such that 
\begin{equation}\label{eq_finale_implica}
    (\forall (e,y)\in P)(\forall q\in \beta(\widetilde{n}(e,y)))(\ell'\cdot\pairing{\pairing{\epsilon(e),\psi(y)},q}\in \gamma(e)\oplus \alpha(y))
\end{equation}

Notice that by definition of $(Y,\psi)\rhd_{(X,\varphi)} (Z,\eta)$  we have that $$(\forall y\in f^{-1}(h(e)))\,(\forall q\in \beta(k(e)(y)))\,(l(e) \cdot\langle \psi(y),q\rangle \in  R(e)\oplus\alpha(y))$$
and since $\widetilde{n}(e,y)=k(e)(y)$, this is equivalent to   
\[(\forall (e,y)\in P)(\forall q\in \beta(\widetilde{n}(y,e)))\,(l(e) \cdot\langle \psi(y),q\rangle \in  R(e)\oplus\alpha(y))\]
On then other hand, by \eqref{eq_sec_verso_implica}, we know that there exists $\ell\in \pca{A}'$ such that
\[ (\forall e\in E)\,(\forall q\in  R(e ))\,(\ell\cdot \pairing{\epsilon(e),q} \in \gamma(e))\]
Hence, combining $\ell\in \pca{A}$ with the fact that $n\function{(E,\epsilon)}{((Y,\psi)\rhd_{(X,\varphi)} (Z,\eta), \psi \rhd_{(X,\varphi)} \eta)}$ is a morphism of partitioned assemblies with a realizer $\overline{r}\in \pca{A}'$ and thus each $l(e)$ can be obtained in a computable way using $\overline{r}$, we can define an element $\ell'$ such that \eqref{eq_finale_implica} is satisfied. 




\end{proof}

\begin{proposition}\label{heytmor} For every arrow $f:(Y,\psi)\rightarrow (X,\varphi)$ in $\partAsm(\pca{A},\subpca{A})$, $\InstRedDoc(f)\function{\InstRedDoc(X,\varphi)}{\InstRedDoc(Y,\psi)} $ is a morphism of Heyting algebras. 
\end{proposition}
\begin{proof}
It is immediate to check that $\InstRedDoc(f)$ preserve maxima and minima. The preservation of binary infima and binary suprema follow essentially from the basic properties of pullbacks and from the fact that the category of partitioned assemblies has stable binary coproducts.

The only construct which requires a more careful analysis is that of Heyting implication.
\end{proof}
\begin{proposition}\label{right}
    For every arrow $f:(Y,\psi)\rightarrow (X,\varphi)$ in $\partAsm(\pca{A},\subpca{A})$, $\InstRedDoc(f)\function{\InstRedDoc(X,\varphi)}{\InstRedDoc(Y,\psi)} $  has a right adjoint $\forall_f\function{\InstRedDoc(Y,\psi)}{\InstRedDoc(X,\varphi)}$, satisfying Beck-Chevalley condition, defined by the following assignment:
\[\forall_{f}(g,\alpha):= ((E,\rho)\xrightarrow{\pi_1}(X,\varphi),\bar{\alpha})\] 
for every $\freccia{(Y',\psi')}{g}{(Y,\psi)}$, $\alpha:Y'\rightarrow \powerset(\pca{A})$ where $$E:=\{(x,k,e)|\,x\in X,e\in \pca{A},\qquad\qquad\qquad\qquad\qquad\qquad\qquad\qquad\qquad\qquad$$
$$ \qquad e\Vdash k:(f^{-1}(x),\psi_{|{f^{-1}(x)}})\rightarrow (Y',\psi')  \text{ in }\partAsm(\pca{A},\pca{A}), g\circ k=\mathsf{id}_{f^{-1}(x)}\},$$
$\rho(x,k,e):=\pairing{\varphi(x),e}$ and 
\[\bar{\alpha} (x,k,e):=\bigcup_{y\in f^{-1}(x)}\{\psi(y)\}\otimes\alpha(k(y)).\]

\end{proposition}
\begin{proof}
    Suppose that $(h,\gamma) \leq \forall_{f}(g,\alpha)$, in $\InstRedDoc (X,\varphi) $, i.e.\ there exists a morphism $m$ of partitioned assemblies such that the diagram
\[\begin{tikzcd}
	& {(X',\varphi')} \\
	{(E,\rho)} & {(X,\varphi)}
	\arrow["m"', dashed, from=1-2, to=2-1]
	\arrow["h", from=1-2, to=2-2]
	\arrow["{\pi_1}"', from=2-1, to=2-2]
\end{tikzcd}\]
commutes and there exists  $\ell\in \mathbb{A}'$ such that 
$$\forall x'\in X'\forall q\in \bar{\alpha}(m(x'))\,(\ell\cdot\langle \varphi'(x'),q\rangle\in \gamma(x')).$$
We have to show that $\InstRedDoc(f) (h,\gamma) \leq (g,\alpha)$. The first step is defining a morphism of partitioned assemblies $\bar{m}$ such that the following diagram
\[\begin{tikzcd}
	& {(P,\pi)} \\
	{(Y',\psi)} & {(Y,\psi)}
	\arrow["{\bar{m}}"', dashed, from=1-2, to=2-1]
	\arrow["{f^*(h)}", from=1-2, to=2-2]
	\arrow["g"', from=2-1, to=2-2]
\end{tikzcd}\]
commutes, where $(P,\pi)$ is the pullback of $f$ and $h$, i.e.\ $P:=\{((x',y)\in X'\times Y| f(y)=h(x')\}$. We define $\bar{m}(x',y):= \pr_2 (m (x'))(y)$ (recall that $\pi_1(m(x'))=h(x')$ by assumption). Moreover, notice the  $g\circ \bar{m}(x',y)=f^*(h)(x',y)=y$ because, by definition of the second component of $E$, $g\circ \pr_2 (x,k,e)=\id_{f^{-1}(x)}$.

The morphism $\bar{m}$ is a morphism of partitioned assemblies, since it is realized by $\lambda x.(\sprPCA(r(\sprPCA x)))(\fprPCA x)$ where $r\Vdash m$. In order to show that there exists $\bar{\ell}\in \mathbb{A}'$ such that
$$\forall p \in P \forall q\in \alpha\circ \bar{m}(p)\,(\bar{\ell}\cdot\langle \pi(p),q\rangle\in \gamma\circ  h^*(f)(p)).$$
it is enough to take
$$\bar{\ell}:=\lambda x.\ell\langle \fprPCA(\fprPCA x),\langle \sprPCA(\fprPCA x),\sprPCA x\rangle \rangle.$$
%

Conversely, let us assume that $\InstRedDoc(f) (h,\gamma) \leq (g,\alpha)$ via a morphism of partitioned assemblies $n\function{(P,\pi)}{(Y',\psi')}$ such that the diagram
\[\begin{tikzcd}
	& {(P,\pi)} \\
	{(Y',\psi)} & {(Y,\psi)}
	\arrow["{n}"', dashed, from=1-2, to=2-1]
	\arrow["{f^*(h)}", from=1-2, to=2-2]
	\arrow["g"', from=2-1, to=2-2]
\end{tikzcd}\]
commutes and an element $\ell'\in \subpca{A}$ such that

\begin{equation}\label{eq_penultima_forall}
    \forall p \in P \forall q\in \alpha\circ n(p)\,(\ell'\cdot\langle \pi(p),q\rangle\in \gamma\circ  h^*(f)(p)).
\end{equation}

We have to show that $(h,\gamma)\leq \forall_f(g,\alpha)$. First, let us consider the morphism of partitioned assemblies $$\widehat{n}:(X',\varphi')\rightarrow (E,\rho)$$
$$x'\mapsto (h(x'),n(x',-),\lambda w.r\langle \varphi'(x'),w\rangle)$$
where $r\Vdash n$.
The function $\widehat{n}$ is a morphism of partitioned assemblies, since a realizer for it is $\lambda u.\langle r'u,\lambda w.r\langle u,w\rangle\rangle$, where $r'$ is a realizer for $h$. Now it is immediate to check that diagram

\[\begin{tikzcd}
	& {(X',\varphi')} \\
	{(E,\rho)} & {(X,\varphi)}
	\arrow["\widehat{n}"', dashed, from=1-2, to=2-1]
	\arrow["h", from=1-2, to=2-2]
	\arrow["{\pi_1}"', from=2-1, to=2-2]
\end{tikzcd}\]
commutes by definition of the first component $\widehat{n}(x')$. Now we have to define an element $\bar{\ell'}\in \pca{A}'$ such that
$\ell\in \mathbb{A}'$ such that 
\begin{equation}\label{eq_ultima_forall}
\forall x'\in X'\forall q\in \bar{\alpha}(\widehat{n}(x'))\,(\bar{\ell'}\cdot\langle \varphi'(x'),q\rangle\in \gamma(x')).
\end{equation}
By definition, we have that
$$
    \bar{\alpha} (\widehat{n}(x'))=\bar{\alpha} (h(x'),n(x',-),\lambda w.r\langle \varphi'(x'),w\rangle) =\bigcup_{y\in f^{-1}h(x')}\{\psi(y)\}\otimes\alpha(n(x',y))
$$
hence, to conclude, it is enough to take $\bar{\ell'}$ as $\lambda u.\ell'\cdot \pairing{\pairing{\fprPCA u,\fprPCA(\sprPCA u)},\sprPCA(\sprPCA u)}$. Combing this choice with the hypothesis \eqref{eq_penultima_forall} we can conclude that \eqref{eq_ultima_forall} holds, and hence that $(h,\gamma)\leq \forall_f(g,\alpha)$.
\end{proof}

\begin{theorem} \label{generic}$\InstRedDoc$ has a generic element.
\end{theorem}
\begin{proof}
As a generic element, we can consider the partitioned assembly $$(\powerset(\powerset(\pca{A}))^{\pca{A}},c_{\kPCA}),$$ where $c_{\kPCA}$ denotes the constant function $\Phi\mapsto \kPCA$, endowed with the element of $\InstRedDoc (\powerset(\powerset(\pca{A}))^{\pca{A}},c_{\kPCA}) $ defined by $(\pi_1,\pi_3)$, with respect to the domain

$$\left(\{(\Phi,a,A)|\,\Phi\in \powerset(\powerset(\mathbb{A}))^{\mathbb{A}},a\in \mathbb{A},A\in \Phi(a)\},\pi_2\right).$$
Let $(f,\alpha)$ be a predicate over $(X,\varphi)$ with $f\function{(Y,\psi)}{(X,\varphi)}$. We have to show that there exists a morphism of partitioned assemblies
\[\chi_{(f,\alpha)}\function{(X,\varphi)}{(\powerset(\powerset(\pca{A}))^{\pca{A}},c_{\kPCA})} \]
(or simply a function because $c_{\kPCA}$ is the constant function) such that 
$$(f,\alpha)=\InstRedDoc_{\chi_{(f,\alpha)}}(\pi_1,\pi_3).$$
Hence, we can consider the function $\chi_{(f,\alpha)}\function{X}{\powerset (\powerset(\pca{A}))^{\pca{A}} }$ sending each $x\in X$ to the function $\chi_{(f,\alpha)}(x)\function{\pca{A}}{\powerset (\powerset(\pca{A}))}$ defined by the assignment 
$$a\mapsto \{\alpha(y)|\,y\in Y, f(y)=x,\psi(y)=a\}.$$

The reindexing $\InstRedDoc_{\chi_{(f,\alpha)}}(\pi_1,\pi_3)$ of the generic element along this morphism can be represented as the first projection  $\pi_1'$ from the partitioned assemblies
$$(Y^*,\psi^*):=(\{(x,\psi(y),\alpha(y))|\,y\in Y, f(y)=x\},(x,\psi(y),\alpha(y))\mapsto \langle  \varphi(x),\psi(y)\rangle)$$
to $(X,\varphi)$ together with the function $\pi_3'\function{Y^*}{\powerset (\pca{A})}$ sending $(x,\psi(y),\alpha(y))$ to $\alpha(y)$.
In order to conclude, we have to show that $(\pi_1',\pi_3')$ is equivalent to $(f,\alpha)$.

It is easy to define a morphism of partitioned assemblies $h\function{(Y,\psi)}{(Y^*,\psi^*)}$ such that the diagram
\[\begin{tikzcd}
	& {(Y,\psi)} \\
	{(Y^*,\psi^*)} & {(X,\varphi)}
	\arrow["h"', dashed, from=1-2, to=2-1]
	\arrow["f", from=1-2, to=2-2]
	\arrow["{\pi_1'}"', from=2-1, to=2-2]
\end{tikzcd}\]
commutes: one just sends every $y\in Y$ to the triple $h(y):=(f(y),\psi(y),\alpha(y))$. This is a morphism of partitioned assemblies since $f$ is so. Moreover, for every $y\in Y$ and for every $q\in \pi_3'\circ h(y)$, i.e.\ for every $q\in \alpha(y)$, we have that $\sprPCA\cdot \langle \psi(y),q\rangle=q\in \alpha(y)$. Thus $(f,\alpha)\leq (\pi_1',\pi_3')$.

Now we show that $(\pi_1',\pi_3')\leq (f,\alpha)$. Using the axiom of choice one can also establish the existence of a function $m\function{(Y^*,\psi^*)}{(Y,\psi)}$ such that the diagram
\[\begin{tikzcd}
	& {(Y^*,\psi^*)} \\
	{(Y,\psi)} & {(X,\varphi)}
	\arrow["m"', dashed, from=1-2, to=2-1]
	\arrow["{\pi_1'}", from=1-2, to=2-2]
	\arrow["f"', from=2-1, to=2-2]
\end{tikzcd}\]
commutes: $m$ sends each triple $(x,a,A)$ in $Y^*$ to an element $y\in Y$ such that $f(y)=x$, $a=\psi(y)$ and $A=\alpha(y)$. Notice that $m$ is a morphism of partitioned assemblies because 
\[\sprPCA\psi^{*}(x,a,A)=\sprPCA\langle \varphi(x),\psi(y)\rangle=\sprPCA\langle \varphi(x),\psi(m(x,a,A))\rangle=\psi(m(x,a,A)).\] 
Finally, for every $(x,a,A)\in Y^{*}$ and every $q\in \alpha(m(x,a,A))$, i.e.\ for every $q\in A$, we have that $\sprPCA\pairing{\pairing{ \varphi(x),\psi(y)}, q} \in \pi_3'(x,a,A)$. Thus, we can conclude that $(\pi_1',\pi_3')\leq (f,\alpha)$.

\end{proof}

\begin{theorem}\thlabel{thm_tripos_ex_wei}
$\InstRedDoc$ is a tripos.
\end{theorem}
\begin{proof}This is a direct consequence of Propositions \ref{heytobj}, \ref{heytmor} and \ref{right}, and Theorems \ref{generic} and \ref{exists}.
\end{proof}

\begin{remark}
    Notice that in \cite{trottavalenti2025}, a doctrine was already introduced such that its fibre on the terminal object corresponds to extended Weihrauch degrees, and such a doctrine was constructed as the \emph{pure} existential completion (freely adding left adjoints just along product projections) of a more basic doctrine. However, the doctrine introduced in \cite{trottavalenti2025} is not a tripos. The main reason is that the pure existential completion is the minimal construction we need, in order to have that the fibres on the terminal abstract the desired degrees. However, it does not add enough structure to obtain a tripos. The full existential completion maintains the same fibre over the terminal object (hence it still abstract extended Weihrauch degrees) but by freely adding \emph{all} the left adjoints we get more structure on the other fibres. Indeed, in this case we obtain a tripos~\thref{thm_tripos_ex_wei}.

\end{remark}

\section{A tripos for extended Weihrauch degrees}
The main purpose of this section is to present a tripos for extended Weihrauch degrees. 

The main intuition is that $\InstRedDoc$ represents a direct categorification of the notion of instance reduction between realizability predicate. Indeed, \thref{thm_eW_1_is_Wei} holds just by definition of $\InstRedDoc(1)$. 

We are going to introduce a second tripos, which will abstract the notion of extended Weihrauch degrees almost by definition, and then prove that this is equivalent to $\InstRedDoc$.

\begin{definition}[extended Weihrauch doctrine]\thlabel{def_ex_deg_doctrine}
    The {\bf extended Weihrauch doctrine} $\doctrine{\partAsm(\pca{A},\subpca{A})}{\EWei}$ is defined as follows: for every object $(X,\varphi)$ of $\partAsm(\pca{A},\subpca{A})$, the elements of $\EWei(X,\varphi)$ are functions $f\function{X\times \pca{A}}{\powerset(\powerset(\pca{A}))}$. For every pair of maps $f,g$ in $\powerset(\powerset(\pca{A}))^{X\times \pca{A}}$, we define $f \EWeiDoctrineOrder g$ if and  only if there exist $\ell_1,\ell_2\in \mathbb{A}'$ such that for every $ (x,a)\in X\times  \mathbb{A}$ such that $f(x,a)\neq \emptyset$
    \begin{itemize}
        \item $\ell_1\cdot \pairing{\varphi(x),a}\downarrow$ and $g(x,\ell_1\cdot \pairing{\varphi(x),a})\neq \emptyset$;
        \item for every $A\in f(x,a)$ there exists $ B\in g(x,\ell_1\cdot \pairing{ \varphi(x),a})$ such that for every $ q\in B$ we have $\ell_2\cdot \pairing{a,q}\in A$.
    \end{itemize}
    As usual, $\EWei (X,\varphi)$ is the posetal reflection of $(\powerset(\powerset(\pca{A}))^{X},\EWeiDoctrineOrder$). The action of $\EWei$ on a morphism of partitioned assemblies $$h\function{(Y,\psi)}{(X,\varphi)}$$
is defined as follows:
$$\EWei_{h}:\EWei(X,\varphi)\rightarrow \EWei(Y,\psi)$$
$$g\mapsto \left((y,a)\mapsto \begin{cases}g(h(y),\sprPCA a)\text{ if }\fprPCA a= \psi(y)\\
\emptyset \text{ otherwise}\end{cases}\right)$$
\end{definition}
 It is now immediate to see that:
\begin{theorem}\thlabel{thm_EWei(1) is ext weih}
    The fibre $\EWei (1)$ corresponds exactly to the extended Weihrauch degrees.
\end{theorem}

Now we are going to show that $\EWei$ and $\InstRedDoc$ are equivalent doctrines, and hence that $\EWei$ is a tripos.  Notice that \thref{thm_ex_descr_wei_doc} can be seen as a fibrational version of \thref{prop_ext_weih_equiv_instance_reduction}, namely of Bauer's result showing that  Weihrauch reductions and instance reductions are equivalent see~\cite[Prop. 8]{Bauer2021}. 

\begin{theorem}\thlabel{thm_ex_descr_wei_doc}

For every partitioned assembly $(X,\varphi)$, we have a natural isomorphism $$\InstRedDoc(X,\varphi)\cong\EWei(X,\varphi).$$
Thus, $\InstRedDoc$ and $\EWei$ are isomorphic doctrines.
\end{theorem}
\begin{proof}
Consider the function $F_{(X,\varphi)}$ sending an element $((Y,\psi)\xrightarrow{f} (X,\varphi),\alpha)$ to $$F_{(X,\varphi)}(f,\alpha):X\times \mathbb{A}\rightarrow \powerset(\powerset(\mathbb{A}))$$
$$(x,a)\mapsto \{\alpha(y)|\,y\in f^{-1}(x),\psi(y)=a\}$$
We start by checking that this function preserves the order, i.e.\ that is a morphism of posets. Hence, let us suppose that $(f,\alpha)\leq (g,\beta)$, where the domain of $f$ is $(Y,\psi)$ and the domain of $g$ is $(Z,\eta)$. By definition, this means that there exists a morphism of partitioned assemblies $h\function{(Y,\psi)}{(Z,\eta)}$ such that $g\circ h=f$ and $\alpha \EWeiOrderElementary \EWeiElementary_h (\beta)$, i.e.\ there exists $\ell\in \mathbb{A}'$ such that
\begin{equation}\label{eq_equivalenza_fibre_1}
    \forall y\in Y\,\forall q\in \beta(h(y))\,\ell\cdot \langle \psi(y),q\rangle\in \alpha(y)
\end{equation}
We can define $\ell_1$ and $\ell_2$ as prescribed by the definition of $F_{(X,\varphi)}(f,\alpha)\EWeiDoctrineOrder F_{(X,\varphi)}(g,\beta)$ as follows: $\ell_1:=\lambda \xi.r(\sprPCA\xi)$ and $\ell_2:=\ell$, where $r$ is a realizer for $h$ to be a morphism of partitioned assemblies. 
Indeed, for every $(x,a)\in X\times \pca{A}$ such that $F_{(X,\varphi)}(f,\alpha)(x,a) \neq \emptyset$ we have that 
$\ell_1\cdot \pairing{\varphi (x),a}\downarrow$, because $f(y)=x$ and $\psi (y)=a$, and $F_{(X,\varphi)}(g,\beta)(x,\ell_1\cdot \pairing{\varphi (x),a})\neq \emptyset$, because $h$ is a morphism of partitioned assemblies such that $f=g\circ h$. Moreover, for every $A\in F(f,\alpha)(x,a) $, i.e.\ for every $\alpha (y)$, where $y\in f^{-1}(x)$ and $\psi (y)=a$, we can choose $B\in F_{(X,\varphi)}(g,\beta)(x,\ell_1\cdot \pairing{\varphi (x),a}$ as $B:=\beta (h(y))$ and this satisfies the condition $\ell_2\cdot \pairing{a,q}\in A$ because of \eqref{eq_equivalenza_fibre_1}. Thus $f\leq g$ implies $F_{(X,\varphi)}(f,\alpha)\EWeiDoctrineOrder F_{(X,\varphi)}(g,\beta)$, so $F_{(X,\varphi)}$ is monotone.

Conversely, we can define a function $G_{(X,\varphi)}$ sending $g:X\times \mathbb{A}\rightarrow \powerset(\powerset(\mathbb{A}))$ to 
$$G(g):=(\pi_1\function{X_g}{X},\pi_3\function{X_g}{\powerset(\pca{A}}))$$
where $\pi_1$ and $\pi_3$ are the first and third projections defined on the partitioned assembly:
$$X_g:=(\{(x,a,A)|A\in g(x,a)\},\varphi_g)$$
where $\varphi_g(x,a,A):=\pairing{\varphi(x),a}$. Now let us suppose that 
$g\EWeiDoctrineOrder h$, i.e.\  
\begin{equation}\label{eq_equivalenza_fibre_2}
    \exists \ell_1,\ell_2\in \mathbb{A}'\forall x\in X\forall a\in \mathbb{A}\forall A\in g(x,a)\exists B\in h(x,\ell_1\cdot \langle \varphi(x),a\rangle)\forall q\in B(\ell_2\cdot \langle a,q\rangle\in A)
\end{equation}
Once an $\ell_1$ and an $\ell_2$ are fixed, using the axiom of choice, we can define a function $(x,a, A)\mapsto \overline{B}$ choosing for every $(x,a,A)$ one of the sets $\overline{B}\in h(x,\ell_1\cdot \langle \varphi(x),a\rangle)$ satisfying the previous condition.
In particular, we can define a define a function 
$$L\function{X_g}{X_h}$$
$$(x,a,A)\mapsto (x,\ell_1\cdot \langle \varphi(x),a\rangle, \overline{B})$$
that is also a morphism of partitioned assemblies. Moreover, by definition, the diagram
\[\begin{tikzcd}
	& {X_g} \\
	{X_h} & X
	\arrow["L"', from=1-2, to=2-1]
	\arrow["{\pi_1}", from=1-2, to=2-2]
	\arrow["{\pi_1'}"', from=2-1, to=2-2]
\end{tikzcd}\]
commutes. Now we have to check that $\pi_3\EWeiOrderElementary \pi_3'\circ L$, i.e.\ that there exists $\ell\in \pca{A}'$ such that
\[\forall (x,a,A)\in X_g,\forall q\in \pi_3'\circ L(x,a,A), \ell \cdot \pairing {\varphi_g(x,a,A),q}\in \pi_3(x,a,A)\]
namely,
\[\forall (x,a,A)\in X_g,\forall q\in \overline{B}, \ell \cdot \pairing {\pairing{\varphi(x),a},q}\in A\]
but easily follows by \eqref{eq_equivalenza_fibre_2} and by  considering $\ell:=\lambda \xi.\ell_2\cdot \langle \sprPCA(\fprPCA\xi) ,\sprPCA\xi\rangle$. Hence we get that $G_{(X,\varphi)}(g)\leq G_{(X,\varphi)}(h)$. Thus also $G_{(X,\varphi)}$ is monotone.

Finally, it is direct to check that both $G_{(X,\varphi)}\circ F_{(X,\varphi)}$ and $F_{(X,\varphi)}\circ G_{(X,\varphi)}$ are identities. Indeed, one can immediately check that 
$$F_{(X,\varphi)}(G_{(X,\varphi)}(g))(x,a)=\{ A\in g(x,a')| \pairing{\varphi (x),a'}=a\}$$ for every $g:X\times \mathbb{A}\rightarrow \powerset(\powerset(\mathbb{A}))$. Hence $F_{(X,\varphi)}(G_{(X,\varphi)}(g)) \EWeiDoctrineOrder g$ via $\ell_1:=\lambda \xi. \sprPCA \sprPCA\xi$, $\ell_2\sprPCA$ and $g\EWeiDoctrineOrder F_{(X,\varphi)}(G_{(X,\varphi)}(g)) $ via $\ell_1:=\kPCA$ and $\ell_2:=\sprPCA$. Thus, $F\circ G$ is the identity.

Similarly, one can prove that $G_{(X,\varphi)}\circ F_{(X,\varphi)}$ is the identity too. The proof that $(f,\alpha)\leq G_{(X,\varphi)}(F_{(X,\varphi)}(f,\alpha))$ is trivial (one can just consider the morphism of partitioned assemblies $y\mapsto (\psi(y),f(y),\alpha(y))$), while in order to show that $G_{(X,\varphi)}(F_{(X,\varphi)}(f,\alpha))\leq (f,\alpha)$ one has to use the axiom of choice.
It remains to prove the naturality, i.e.\ one only needs to check that for every $k\function{(X,\varphi)}{(X',\varphi')}$ the square

\[\begin{tikzcd}
	\InstRedDoc (X',\varphi') & \InstRedDoc (X,\varphi)  \\
	\EWei (X',\varphi')  & \EWei (X,\varphi)
	\arrow["\InstRedDoc_k", from=1-1, to=1-2]
	\arrow["F_{ (X',\varphi') }"', from=1-1, to=2-1]
	\arrow["F_{(X,\varphi)} ", from=1-2, to=2-2]
	\arrow["\EWei_k"', from=2-1, to=2-2]
\end{tikzcd}\]
commutes. By definition, for every element $(f:(Y,\psi)\rightarrow (X',\varphi'),\alpha)$ of $\InstRedDoc (X',\varphi')$, one has that
$$\EWei_k(F_{(X',\varphi')}(f,\alpha))(x,a)=\begin{cases}F_{(X',\varphi')}(f,\alpha)(k(x),\sprPCA a)\text{ if }\fprPCA a= \varphi(x)\\
\emptyset \text{ otherwise}\end{cases}$$
that is
$$\EWei_k(F_{(X',\varphi')}(f,\alpha))(x,a)=\begin{cases}\{\alpha(y)|\,y\in f^{-1}(k(x)),\psi(y)=\sprPCA a\}\text{ if }\fprPCA a= \varphi(x)\\
\emptyset \text{ otherwise}\end{cases}$$
On the other hand, $(F_{(X,\varphi)} \circ \InstRedDoc_k(f,\alpha))(x,a) $ is defined as follows: first, by definition we have that 
$$\InstRedDoc_k(f,\alpha):=(k^*f\function{(P,\pi)}{(X,\varphi),\alpha \circ f^*k)}$$
where $P:=\{(x,y)\in X\times Y| f(y)=k(x)\}$. Hence, we have that 
\[F_{(X,\varphi)} (\InstRedDoc_k(f,\alpha))(x,a) =\{\alpha \circ f^*k(x,y)| (x,y)\in (k^*f)^{-1} (x), \pi(x,y)=a\}\]
i.e.\ 
\[F_{(X,\varphi)} (\InstRedDoc_k(f,\alpha))(x,a) =\{\alpha (y)| y\in f^{-1} (k(x)), \pairing{\varphi (x),\psi (y)}=a\}.\]
Then we can conclude that $\EWei_k F_{(X',\varphi')}(f,\alpha))=F_{(X,\varphi)} (\InstRedDoc_k(f,\alpha))$.
\end{proof}
We recall that the dialectica construction can be presented categorically in terms of existential and universal completion of doctrines or fibrations, see \cite{hofstra2011,dePaiva1989dialectica,trotta23TCS}.
Combining the previous result with \thref{prop:eiR is universal com} we can prove the extended Weihrauch doctrine is an instance of the \emph{full or generalized dialectica completion}. 
\begin{corollary}
    The doctrine $\doctrine{\partAsm(\pca{A},\subpca{A})}{\EWei}$ is the (full) dialectica completion $(\mathsf{t}^{\forall})^{\exists}$ the trivial doctrine $\doctrine{\partAsm(\pca{A},\subpca{A})}{\mathsf{t}}$.
\end{corollary}

\section{A topos for extended Weihrauch degrees }
The main purpose of this section is to define a topos for extended Weihrauch degrees and show how this topos is related to a (relative) realizability topos~\cite{BIRKEDAL2002115,hyland1982effective}. 

The tool we are going to use is the so-called \emph{tripos-to-topos} construction~\cite{pitts02,hyland89}, namely the free construction producing a topos starting from a tripos. 

\subsection{The tripos-to-topos}
The tripos-to-topos construction has been introduced in order to show that, from an abstract perspective, realizability and localic toposes can be regarded as instances of the same construction. 

The main intuition is that the topos defined by the tripos-to-topos is given by taking partial equivalence relations and functional relations according to the logic of the  starting tripos.

\bigskip
\noindent
\textbf{Tripos-to-topos.} Given a tripos $\doctrine{\mC}{P}$, the category $\mathsf{T}_P$ consists of:
 
\medskip
\noindent 
\textbf{objects:} pairs $(A,\rho)$ such that $\rho \in P(A\times A)$  satisfies

\begin{itemize}

\item \emph{symmetry:} $\rho\leq P_{\angbr{\pr_2}{\pr_1}}(\rho)$;
\item \emph{transitivity:} $P_{\angbr{\pr_1}{\pr_2}}(\rho)\wedge P_{\angbr{\pr_2}{\pr_3}}(\rho)\leq P_{\angbr{\pr_1}{\pr_3}}(\rho)$, where $\pr_i$ are the projections from $A\times A\times A$;
\end{itemize} 
\textbf{arrows:} $\freccia{(A,\rho)}{\phi}{(B,\sigma)}$ are objects $\phi\in P(A\times B)$ such that
\begin{enumerate}
\item $\phi\leq P_{\angbr{\pr_1}{\pr_1}}(\rho)\wedge P_{\angbr{\pr_2}{\pr_2}}(\sigma)$;
\item $P_{\angbr{\pr_1}{\pr_2}}(\rho)\wedge P_{\angbr{\pr_1}{\pr_3}}(\phi)\leq P_{\angbr{\pr_2}{\pr_3}}(\phi)$ where $\pr_i$ are projections from $A\times A\times B$; 
\item $P_{\angbr{\pr_2}{\pr_3}}(\sigma)\wedge P_{\angbr{\pr_1}{\pr_2}}(\phi)\leq P_{\angbr{\pr_1}{\pr_3}}(\phi)$ where $\pr_i$ are projections from $A\times B\times B$; 
\item $P_{\angbr{\pr_1}{\pr_2}}(\phi)\wedge P_{\angbr{\pr_1}{\pr_3}}(\phi)\leq P_{\angbr{\pr_2}{\pr_3}}(\sigma)$ where $\pr_i$ are projections from $A\times B\times B$; 
\item $P_{\Delta_A}(\rho)\leq \exists_{\pr_1}(\phi)$ where $\pr_1$ is the first projection from $A\times B$.
\end{enumerate}
\textbf{composition:} 
If $\phi:(A,\rho)\rightarrow (B,\sigma)$ and $\psi:(B,\sigma)\rightarrow (C,\eta)$ are arrows, their composition is defined as $\exists_{\langle \pi_1,\pi_3\rangle}(P_{\langle \pi_1,\pi_2\rangle}(\phi)\wedge P_{\langle \pi_2,\pi_3\rangle}(\psi))$, where $\pi_i$ are projections from $A\times B\times C$.
\medskip
\noindent

From a logical perspective, the five conditions describing the arrows in the tripos-to-topos construction can be interpreted as follows: 1. ensures that the relation $\phi$ has the correct domain and codomain; 2.-3. express that $\phi$ is compatible with equality; 4. states that $\phi$ is functional; and 5. that it is total.
%

\begin{theorem}
	Let $\doctrine{\mC}{P}$ be a tripos. Then $\mathsf{T}_P$ is a topos.
\end{theorem}
The following examples are discussed in \cite{pitts02,hyland89}, and they show that realizability and localic toposes can be presented as instances of the same abstract construction.
\begin{example}
The topos $\mathsf{T}_{\mathsf{A}}$ associated to the localic tripos $\doctrine{\set}{\mathsf{A}^{(-)}}$ is equivalent to the category of sheaves $\mathsf{sh}(\mathsf{A})$ of the locale $\mathsf{A}$.
\end{example}
\begin{example}\thlabel{ex_RT_as_tripos_to_topos}
Given a PCA $\pca{A}$, the tripos-to-topos $\mathsf{T}_{\mathcal{P}}$ of the realizability tripos $\doctrine{\set}{\mathcal{P}}$ is equivalent to the realizability topos $\mathsf{RT}(\pca{A})$.
\end{example}

We recall that the problem of understanding the universal property of the tripos-to-topos construction is non-trivial, and depends on the categorical setting we consider. Indeed, one can study this construction by considering the 2-categories toposes and triposes (with certain morphisms). Frey proved in~\cite{Frey2011,FREY2015} that in this setting the tripos-to-topos does not raise to an ordinary bi-adjunction, but to a \emph{special bi-adjunction}.

Maietti and Rosolini considered in~\cite{maiettirosolini13b,maiettirosolini13a,maiettipasqualirosolini} the tripos-to-topos in the general context of \emph{elementary existential doctrines}, showing that in this setting the construction enjoys the universal property of being an \emph{exact completion} of an elementary existential doctrine. In particular, it gives rise to a bi-adjunction between the 2-category of existential doctrines $\mathsf{ED}$ and the 2-category $\Excat$ of exact categories and exact functors. 

Again, we recall that in this work, for the sake of clarity, we use the word ``existential" to refer to doctrines with \emph{all} left adjoints satisfying (BCC) and whose base category has finite limits. Elementary and existential doctrines are more general doctrines, whose base category has finite products, and whose left adjoints satisfy (BCC) just along product projections and diagonals arrows. 

\begin{theorem}[Exact completion]\thlabel{theorem maietti rosolini pasquali exact comp}
	The assignment $P\mapsto \mathsf{T}_{P}$ extends to a 2-functor 
	\[\begin{tikzcd}
		\ED && \Excat
		\arrow[""{name=0, anchor=center, inner sep=0}, "\mathsf{T}_{(-)}",curve={height=-14pt}, from=1-1, to=1-3]
		\arrow[""{name=1, anchor=center, inner sep=0}, "\Sub_{(-)}",curve={height=-14pt}, hook',from=1-3, to=1-1]
		\arrow["\dashv"{anchor=center, rotate=-90}, draw=none, from=1, to=0]
	\end{tikzcd}\]
	from the 2-category $\ED$ of existential doctrines to the 2-category $\Excat$ of exact categories, and it is left adjoint to the functor sending an exact category $\mC$ to the doctrine $\Sub_{\mC}$ of its subobjects.
	\end{theorem}
The tripos-to-topos can be seen, in particular, as a generalization of the exact completion~\cite{CARBONI199879} of a category with finite limits. We recall the following example from \cite{maiettirosolini13b,maiettirosolinipasquali2019}:
    \begin{example}\thlabel{ex: exact comp of weak sub 1}  
	The exact completion  $\exlex{\mC}$ of a category $\mC$  with finite limits happens to be equivalent to the exact completion (tripos-to-topos) $\mathsf{T}_{\Psi_{\mC}}$ of the doctrine $\doctrine{\mC}{\Psi_{\mC}}$ of weak subobjects of $\mC$.
	\end{example}

\begin{remark}\thlabel{rem_real_topos_as_ex_lex}
We have seen in \thref{ex_RT_as_tripos_to_topos} that, given a PCA $\pca{A}$, the realizability topos $\mathsf{RT}[\pca{A}]$ is exactly the tripos-to-topos of the realizability tripos. And this works also for relative realizability toposes $\mathsf{RT}[\pca{A},\subpca{A}]$.

But notice that there is a second tripos that we can use to define realizability toposes as tripos-to-topos. Indeed, Robinson and Rosolini proved in~\cite{robinsonrosolini90} that realizability toposes are equivalent to the exact completion of the category of partitioned assemblies. Therefore, combining this result with  \thref{ex: exact comp of weak sub 1} we have that

\[
\mathsf{RT}[\pca{A},\pca{A}'] \equiv \mathsf{T}_{\Psi_{\partAsm (\pca{A},\subpca{A})}}
    \]
\end{remark}
\subsection{The topos of extended Weihrauch degrees}
Now we have introduced all the tools we need to define a topos for extended  Weihrauch degrees:

\begin{definition}\thlabel{def_wei_topos}
    We define the \textbf{topos of extended Weihrauch degrees} as the topos $\mathsf{EW}[\pca{A},\pca{A}']:=\mathsf{T}_{\EWei}$ obtained by applying the tripos-to-topos to the tripos $\doctrine{\partAsm(\pca{A},\subpca{A})}{\EWei}$.
\end{definition}

\textbf{Notation}: for $\rho:X\times \mathbb{A}\rightarrow \powerset\powerset(\pca{A}) $ with $(X,\varphi)$ a partitioned assembly we introduce the following restriction function notation:
$$\rho_{|(X,\varphi)}(x,a):=\begin{cases} \rho(x,\sprPCA a) &\text{ if } \fprPCA a =\varphi (x) \\
\emptyset &\text{ otherwise} \end{cases}$$
and we simply write $\rho_{|}$ where the partitioned assembly is clear from the context.\\

The topos $\mathsf{EW}[\pca{A},\pca{A}']$ can hence be more explicitly presented as follows.\\

\textbf{objects:} pairs $((X,\varphi),\rho)$ such that $\rho\function{(X\times X) \times \pca{A}}{\powerset\powerset (\pca{A})}$  satisfies

\begin{itemize}

\item \emph{symmetry:} $\rho((x_1,x_2),a)\EWeiDoctrineOrder \rho_{|}((x_2,x_1),a)$ where restrictions are w.r.t.\ $(X,\varphi)\times (X,\varphi)$;
\item \emph{transitivity:} $\rho_{|}((x_1,x_2),a)\wedge \rho_{|}((x_2,x_3),a)\EWeiDoctrineOrder  \rho_{|}((x_1,x_3),a)$ where restrictions are w.r.t.\ $(X,\varphi)\times (X,\varphi)\times (X,\varphi)$.
\end{itemize} 

\textbf{arrows:} $\phi:((X,\varphi),\rho)\rightarrow ((Y,\psi),\eta)$  are 
functions $\phi:(X\times Y)\times \mathbb{A}\rightarrow \powerset\powerset(\pca{A})$ such that:

\begin{itemize}

\item  $\phi(x,y,a)\EWeiDoctrineOrder \rho_{|}((x,x),a)\wedge \eta_{|}((y,y),a)$ where restrictions are w.r.t.\ $(X,\varphi)\times  (Y,\psi)$;
\item  $\rho_{|}((x_1,x_2),a)\wedge \phi_{|}((x_1,y),a)\EWeiDoctrineOrder  \phi_{|}((x_2,y),a)$ where restrictions are w.r.t.\ $(X,\varphi)\times (X,\varphi)\times  (Y,\psi)$;
\item  $\eta_{|}((y_1,y_2),a)\wedge \phi_{|}((x,y_1),a)\EWeiDoctrineOrder  \phi_{|}((x,y_2),a)$ where restrictions are w.r.t.\ $(X,\varphi)\times (Y,\psi)\times  (Y,\psi)$;
\item $\phi_{|}((x,y_1),a)\wedge \phi_{|}((x,y_2),a)\EWeiDoctrineOrder  \eta_{|}((y_1,y_2),a)$ where restrictions are w.r.t.\ $(X,\varphi)\times (Y,\psi)\times  (Y,\psi)$;
\item $\rho_{|}((x,x),a)\EWeiDoctrineOrder  \exists_{\pi_1}\phi((x,y),a)$ where restrictions are w.r.t.\ $(X,\varphi)$.
\end{itemize} 

Notice that in the previous items, we used notations of the form $f(a,b,c)$ to indicate the function $(a,b,c)\mapsto f(a,b,c)$. Thus $\wedge$ denoted the infimum in the appropriate fibre of the tripos $\EWei$ and $\exists_{\pi_1}$ the appropriate left adjoint there.

\begin{theorem}\thlabel{thm_main}
    Let $\pca{A}$ be a PCA and $\pca{A}'$ be an elementary sub-PCA of $\pca{A}$. Then the (relative) realizability topos $\mathsf{RT}[\pca{A},\pca{A}']$ can be embedded into the topos of extended Weihrauch degrees $\mathsf{EW}[\pca{A},\pca{A}']$  and this embedding has a left adjoint

\[\begin{tikzcd}
	\mathsf{RT}[\pca{A},\pca{A}'] && \mathsf{EW}[\pca{A},\pca{A}']
	\arrow[""{name=0, anchor=center, inner sep=0}, curve={height=12pt}, hook, "R"',from=1-1, to=1-3]
	\arrow[""{name=1, anchor=center, inner sep=0}, curve={height=12pt}, "L"',from=1-3, to=1-1]
	\arrow["\dashv"{anchor=center, rotate=-90}, draw=none, from=1, to=0]
\end{tikzcd}\]
such that $L\circ R$ is the identity. Moreover, both the functors $R$ and $L$ are morphisms of exact categories.
\end{theorem}
\begin{proof}
    Observe that we have an action of doctrines 
\[\begin{tikzcd}
	\partAsm (\pca{A},\pca{A}')^{\op} \\
	&& \mathsf{InfSl} \\
	\partAsm (\pca{A},\pca{A}')^{\op}
	\arrow[""{name=0, anchor=center, inner sep=0}, "\mathsf{t}", from=1-1, to=2-3]
	\arrow["\id"',from=1-1, to=3-1]
	\arrow[""{name=1, anchor=center, inner sep=0}, "\EWeiElementary"',from=3-1, to=2-3]
	\arrow[""{name=2, anchor=center, inner sep=0}, curve={height=-12pt}, shorten <=6pt, shorten >=6pt, hook', "r",from=0, to=1]
	\arrow[""{name=3, anchor=center, inner sep=0}, curve={height=-12pt}, shorten <=6pt, shorten >=6pt, "l",from=1, to=0]
	\arrow["\dashv"{anchor=center}, draw=none, from=3, to=2]
\end{tikzcd}\]
where $\doctrine{\partAsm [\pca{A},\pca{A}']}{\mathsf{t}}$ is the doctrine whose fibres have just one element, as defined in \thref{ex_weaksub_is_full_ex_comp}. The natural transformation $r$ sends the unique element into the top element, while $l$ sends every element of a fibre into the unique element of the fibre of $\mathsf{t}$. In particular, $l\circ r$ is the identity.

Then we can apply to this adjunction the full existential completion, and we obtain an adjunction of existential doctrines
\[\begin{tikzcd}
	\partAsm (\pca{A},\pca{A}')^{\op} \\
	&& \mathsf{InfSl} \\
	\partAsm (\pca{A},\pca{A}')^{\op}
	\arrow[""{name=0, anchor=center, inner sep=0}, "\Psi_{\partAsm (\pca{A},\pca{A}')}", from=1-1, to=2-3]
	\arrow["\id"',from=1-1, to=3-1]
	\arrow[""{name=1, anchor=center, inner sep=0}, "\EWei "',from=3-1, to=2-3]
	\arrow[""{name=2, anchor=center, inner sep=0}, curve={height=-12pt}, shorten <=6pt, shorten >=6pt, hook', "r",from=0, to=1]
	\arrow[""{name=3, anchor=center, inner sep=0}, curve={height=-12pt}, shorten <=6pt, shorten >=6pt, "l",from=1, to=0]
	\arrow["\dashv"{anchor=center}, draw=none, from=3, to=2]
\end{tikzcd}\]
because by \thref{ex_weaksub_is_full_ex_comp} we have that $\mathsf{t}^{\exists}\cong \Psi_{\partAsm (\pca{A},\pca{A}')}$ and by \thref{thm_ex_descr_wei_doc} we have that $\EWeiElementary^{\exists}= \InstRedDoc \cong \EWei$. Therefore, we are in the right setting for applying the tripos-to-topos construction, obtaining, by \thref{theorem maietti rosolini pasquali exact comp} an adjunction of exact categories
\[\begin{tikzcd}
	\mathsf{T}_{\Psi_{\partAsm (\pca{A},\pca{A}')}} && \mathsf{T}_{\EWei}
	\arrow[""{name=0, anchor=center, inner sep=0}, curve={height=12pt}, hook, "R"',from=1-1, to=1-3]
	\arrow[""{name=1, anchor=center, inner sep=0}, curve={height=12pt}, "L"',from=1-3, to=1-1]
	\arrow["\dashv"{anchor=center, rotate=-90}, draw=none, from=1, to=0]
\end{tikzcd}\]
Now, by \thref{rem_real_topos_as_ex_lex}, we have that $\mathsf{RT}[\pca{A},\pca{A}'] \equiv \mathsf{T}_{\Psi_{\partAsm (\pca{A},\pca{A})'}}$, and by \thref{def_wei_topos}, we have that $\mathsf{EW}[\pca{A},\pca{A}']=\mathsf{T}_{\EWei}$. Hence we obtain the adjunction 
\[\begin{tikzcd}
	\mathsf{RT}[\pca{A},\pca{A}'] && \mathsf{EW}[\pca{A},\pca{A}']
	\arrow[""{name=0, anchor=center, inner sep=0}, curve={height=12pt}, hook, "R"',from=1-1, to=1-3]
	\arrow[""{name=1, anchor=center, inner sep=0}, curve={height=12pt}, "L"',from=1-3, to=1-1]
	\arrow["\dashv"{anchor=center, rotate=-90}, draw=none, from=1, to=0]
\end{tikzcd}\]
\end{proof}
Notice that the adjunction presented in \thref{thm_main} is in particular a geometric embedding of toposes. Hence, given the well-known correspondence between geometric embeddings (i.e.\ geometric morphisms whose counit is an iso) and Lawvere-Tierney topologies (e.g. see \cite[Cor.~7, Sec. VII]{SGL}) we obtain the following corollary:
\begin{corollary}\thlabel{cor_sheaves_real_top}
    The relative realizability topos $\mathsf{RT}[\pca{A},\pca{A}'] $ is equivalent to a topos of $j$-sheaves $\mathsf{sh}_j(\mathsf{EW}[\pca{A},\pca{A}'])$ for a certain Lawvere-Tierney topology $j$ over $ \mathsf{EW}[\pca{A},\pca{A}']$.
\end{corollary}
An interesting application of \thref{thm_main} regards the posets of subobjects over the terminals of the two toposes. Indeed, we have an adjunction
\[\begin{tikzcd}
	\Sub_{\mathsf{RT}[\pca{A},\pca{A}']}(1) && \Sub_{\mathsf{EW}[\pca{A},\pca{A}']}(1)
	\arrow[""{name=0, anchor=center, inner sep=0}, curve={height=12pt}, hook, "R"',from=1-1, to=1-3]
	\arrow[""{name=1, anchor=center, inner sep=0}, curve={height=12pt}, "L"',from=1-3, to=1-1]
	\arrow["\dashv"{anchor=center, rotate=-90}, draw=none, from=1, to=0]
\end{tikzcd}\]
between the subobjects of $1$ of the Kleene-Vesley topos and the subobjects of $1$ of the extended Weihrauch degrees topos.

Now, it is well-known that the subobjects over the terminal object of a tripos-to-topos correspond to the elements of the fibre over the terminal object of the generating tripos. Hence, combining this result with \thref{thm_EWei(1) is ext weih} and \thref{def:full medvedev doctrine} we have the following corollary:

\begin{corollary}\thlabel{last cor}
Let $\pca{A}$ be a PCA, and $\pca{A}'$ be an elementary sub-PCA. Then we have adjunction:
 \[\begin{tikzcd}
	\MedvedevDoctrine(1)^{\op} && \EWei (1)
	\arrow[""{name=0, anchor=center, inner sep=0}, curve={height=12pt}, hook, "R"',from=1-1, to=1-3]
	\arrow[""{name=1, anchor=center, inner sep=0}, curve={height=12pt}, "L"',from=1-3, to=1-1]
	\arrow["\dashv"{anchor=center, rotate=-90}, draw=none, from=1, to=0]
\end{tikzcd}\]
 In particular, the dual of the poset of Medvedev degrees embeds in the poset of extended Weihrauch degrees, and this embedding has a left adjoint.
\end{corollary}
We conclude by providing an explicit description  of such an instance of the geometric embedding $L\dashv R$ and the composition $RL$, which also helps in having an intuition of the features of the topology $j$ presented in \thref{cor_sheaves_real_top}. 

Let us consider an extended Weihrauch predicate $f\function{\pca{A}}{\powerset\powerset(\pca{A})}$. Then we have that
\[L(f):=\{a\in \pca{A}| f(a)\neq \emptyset\}\]
Now, let us consider a subset $\Lambda\subseteq \pca{A}$. Then we define the function $R(\Lambda)\function{\pca{A}}{\powerset\powerset(\pca{A})}$
\[R(\Lambda)(a):=\begin{cases}
    \{\emptyset\} \text{ if } a\in \Lambda\\
    \emptyset \text{ otherwise}
\end{cases}\]
Then, given an  extended Weihrauch predicate $f\function{\pca{A}}{\powerset\powerset(\pca{A})}$, $RL$ associates a predicate $RL(f)\function{\pca{A}}{\powerset\powerset(\pca{A})}$ with the same inputs as $f$, but with no possible outputs:
\[
RL(f)(a) = \begin{cases}
    \{\emptyset\} \text{ if } f(a)\neq \emptyset\\
    \emptyset \text{ otherwise}
\end{cases}
\]
In particular, if $f \neq 0$, then $RL(f)$ is not a $\neg\neg$-dense extended Weihrauch predicate in the sense of Bauer \cite{Bauer2021}.

This explicit description allows us to conclude the embedding presented in \thref{last cor} coincides with the embedding presented in \cite[Lem.~5.6]{HiguchiPauly} (when we consider $\pca{A}:=\mathbb{N}^{\mathbb{N}}$).

\section*{Conclusions and future work}
We have presented here a new tripos abstracting the notion of extended Weihrauch reducibility and the corresponding topos obtained via the tripos-to-topos construction. Then, we have employed the existential completion to establish the formal connection between toposes of extended Weihrauch degrees and (relative) realizability toposes. These happen to be toposes of $j$-sheaves for certain Lawevere-Tierney topologies over the toposes of extended Weihrauch degrees. 

In future work, we aim to start studying in detail the logic inside this new topos we introduced and its categorical properties. We expect it should enjoy some of the main categorical properties of (relative) realizability toposes, such as being an exact completion of a lex category. 

Finally, we intend to explore connections with other categorical notions related to extended Weihrauch reducibility, such as \emph{containers}, see \cite{pradic2025s}, and to study potential connections with linear logic. Indeed, two tensor products are widely studied  in the context of Weihrauch reducibility, see e.g. \cite{KIHARA_MARCONE_PAULY_2020,HiguchiPauly,pradic2025s}. Hence, it would be interesting to study how these notions enrich the internal language of the topos generated by the tripos of extended Weihrauch degrees.

\subsection*{Acknowledgements}
We would like to thank Maria Emilia Maietti and Francesco Ciraulo for ideas and suggestions on ongoing joint work, and Manlio Valenti for helpful discussions and comments concerning Weihrauch degrees and their variants. We thank Cécilia Pradic for observing and suggesting that the elementary instance reducibility doctrine is a universal completion, as stated in \thref{prop:eiR is universal com}. 
Finally, we also thank the anonymous referees for their careful reading of the paper and many valuable suggestions.
\bibliographystyle{plain}
\bibliography{references}

\end{document}